\documentclass{article}

\usepackage{amsmath}%
\usepackage{lmodern}
\usepackage[T1]{fontenc}
\usepackage[babel=true]{microtype}

\usepackage{mathrsfs}
\usepackage{textcomp}
\usepackage{stmaryrd}
\usepackage{wasysym}
\usepackage{amsmath}
\usepackage{amsxtra}%
\usepackage{amsfonts}%
\usepackage{amssymb}%
\usepackage{amsthm}
\usepackage{graphicx}
\usepackage{epstopdf}
\usepackage{color,hyperref}
\hypersetup{colorlinks,breaklinks,
linkcolor=blue,urlcolor=blue,
anchorcolor=blue,citecolor=blue}
\usepackage[dvipsnames]{xcolor}
\usepackage{subfigure}
\usepackage{appendix}
\usepackage{enumerate}
\usepackage{enumitem}
\usepackage[margin=2.54cm]{geometry}
\usepackage{tikz}

\newcommand{\ImagePath}[1]{#1}

\newcommand{\figurescale}{1}

\newcommand{\eps}{\varepsilon}

\newcommand{\N}{\mathbb{N}}
\newcommand{\R}{\mathbb{R}}

\newcommand{\C}{\mathbb{C}}

\newcommand{\mcl}{\mathcal{L}}

\def\XXint#1#2#3{{\setbox0=\hbox{$#1{#2#3}{\int}$ }
\vcenter{\hbox{$#2#3$ }}\kern-.6\wd0}}

\newcommand*{\pscal}[1]{\langle #1\rangle}
\newcommand*{\Cal}[1]{\mathcal{#1}}
\newcommand*{\Rm}[1]{\mathrm{#1}}

\newcommand*{\norme}[1]{\lVert#1\rVert}

\newcommand*{\absolu}[1]{\lvert#1\rvert}
\newcommand*{\transpose}[1]{#1^\mathrm{T}}

\DeclareMathOperator*{\real}{\mathrm{Re}}
\DeclareMathOperator*{\imag}{\mathrm{Im}}
\DeclareMathOperator{\range}{im}
\DeclareMathOperator*{\Span}{Span}		
\DeclareMathOperator*{\fred}{fred}
\renewcommand*{\d}{\mathrm{d}}
\newcommand*{\lra}{\longrightarrow}

\newtheorem{thm}{Theorem}
\newtheorem{prop}{Proposition}
\newtheorem*{thm*}{Theorem}

\newtheorem{lemma}[prop]{Lemma}						
\newtheorem{corollary}[prop]{Corollary}				

\numberwithin{equation}{section}
\numberwithin{prop}{section}

\theoremstyle{definition}
\newtheorem{rmk}[prop]{Remark}

\newcommand\blfootnote[1]{%
\begingroup
\renewcommand\thefootnote{}\footnote{#1}%
\addtocounter{footnote}{-1}%
\endgroup
}

\author{Montie Avery\footnote{Corresponding author. avery142@umn.edu. School of Mathematics, University of Minnesota, 206 Church St. SE, Minneapolis, MN, 55455, USA.} \and Louis Gar\'enaux\footnote{louis.garenaux@math.univ-toulouse.fr. Institut de Math\'ematiques de Toulouse, UMR 5219, Universit\'e de Toulouse CNRS, UPS IMT, F-31062 Toulouse Cedex 9, France.}}
\title{Spectral stability of the critical front in the extended Fisher-KPP equation}

\begin{document}
\maketitle

\begin{abstract}
\noindent We revisit the existence and stability of the critical front in the extended Fisher-KPP equation, refining earlier results of Rottsch\"afer and Wayne \cite{RottschaferWayne} which establish stability of fronts without identifying a precise decay rate. Our main result states that the critical front is marginally spectrally stable, with essential spectrum touching the imaginary axis but with no unstable point spectrum. Together with the recent work of Avery and Scheel \cite{AveryScheelSelection, AveryScheel}, this establishes both sharp stability criteria for localized perturbations to the critical front, as well as propagation at the linear spreading speed from steep initial data, thereby extending front selection results beyond systems with a comparison principle. Our proofs are based on far-field/core decompositions which have broader use in establishing robustness properties and bifurcations of invasion fronts. 
\end{abstract}



\blfootnote{Declaration of competing interests: none.}

\section{Introduction}

\subsection{Background and main results}
The extended Fisher-KPP equation 
\begin{align}
u_t = - \delta^2 u_{xxxx} + u_{xx} + f(u), \quad f(0) = f(1) = 0, \label{e: EFKPP}
\end{align}
is a fundamental model for understanding the dynamics of invasion fronts in systems without comparison principles \cite{DeeSaarloos}, and may further be derived as an amplitude equation near certain co-dimension 2 bifurcations in reaction-diffusion systems \cite{RottschaferDoelman}. Indeed, while rigorous results on front propagation from steep initial data are typically limited to equations with comparison principles, the \textit{marginal stability conjecture} predicts that invasion speeds in spatially extended systems are universally predicted by marginal spectral stability of an associated invasion front \cite{vanSaarloos}. In the current setting, such invasion fronts solve the traveling wave equation
\begin{align}
0 = -\delta^2 q'''' + q'' + c q' + f(q), \quad q(-\infty) = 1, \quad q(\infty) =0. \label{e: EFKPP fronts}
\end{align}
The review paper \cite{vanSaarloos} presents many examples in which this conjectured behavior is observed in systems without comparison principles through numerical simulations, physical experiments, and formal asymptotic analysis. The lack of a comparison principle is essential to much of the interesting dynamics explored in \cite{vanSaarloos}, in which invasion fronts select features of periodic patterns generated in their wake. Concurrent to the present work, the first author and Scheel gave a rigorous proof of the marginal stability conjecture for unpatterned invasion in higher order parabolic systems, identifying precise spectral criteria which lead to selection of critical pulled fronts \cite{AveryScheelSelection}. The present work establishes that these spectral assumptions hold for \eqref{e: EFKPP} for $\delta$ sufficiently small, thereby establishing front selection in the absence of comparison principles and making progress towards understanding the dynamics of pattern forming fronts explored in \cite{vanSaarloos}. 

Here we assume $f$ is of Fisher-KPP type: $f(0) = f(1) = 0$, $f'(0) > 0$, $f'(1) < 0$, and for instance $f''(u) < 0$ for all $u \in (0,1)$; see Section \ref{s: remarks} for comments on this last assumption. In this case, the marginal stability conjecture predicts that strongly localized initial data in \eqref{e: EFKPP} propagate with the \textit{linear spreading speed} $c_*(\delta)$, a distinguished speed for which solutions to the linearization
\begin{align*}
u_t = -\delta^2 u_{xxxx} + u_{xx} + c u_x + f'(0) u
\end{align*}
generically grow exponentially pointwise for $c < c_* (\delta)$ but decay for $c > c_*(\delta)$. The linear spreading speed may be more precisely characterized by the location of simple pinched double roots of the associated dispersion relation; see below for details. Our first result establishes the existence of a critical front traveling with the linear spreading speed, which was previously proved by Rottsch\"afer and Wayne using geometric singular perturbation theory \cite{RottschaferWayne}. 
\begin{thm}[Existence of the critical front] \label{t: existence}
For $\delta$ sufficiently small and for $c = c_*(\delta)$, there exists a smooth traveling front $q_*$ solving \eqref{e: EFKPP fronts}, such that 
\begin{align*}
q_*(x; \delta) = (\mu(\delta) + x) e^{-\eta_* (\delta) x} + \mathrm{O}(e^{-(\eta_*(\delta) + \eta) x}), 
\hspace{4em}
x \to \infty
\end{align*}
for some $\eta > 0$, where $\mu (\delta) = 1 + \mathrm{O}(\delta)$ and $\eta_*(\delta) = \sqrt{f'(0)} + \mathrm{O}(\delta)$. Moreover, $q_*(\cdot; \delta)$ depends continuously on $\delta$, uniformly in space. 
\end{thm}
Our proof is based on a far-field/core decomposition, relying only on basic Fredholm properties of the linearization about the critical front for $\delta = 0$ together with explicit preconditioners which regularize the singular perturbation. We believe our methods have further utility in describing bifurcations from pushed to pulled front propagation as well as analyzing invasion fronts in nonlocal equations. We also mention that the existence of both invasion fronts and fronts connecting two stable states in fourth order parabolic equations, including the extended Fisher-KPP equation with $\delta$ not necessarily small, was established in \cite{Berg_Hulshof_Vandervorst_01} using topological arguments.

Perturbations $v(t, x - c_*(\delta) t) = u(t,x) + q_*(x - c_*(\delta) t; \delta)$ of the critical front in \eqref{e: EFKPP} solve 
\begin{align}
v_t = \mathcal{A}(\delta) v + f(q_* + v) - f(q_*) - f'(q_*) v, \label{e: perturbation eqn}
\end{align}
where $\Cal{A}(\delta) : H^4(\R) \subset L^2(\R) \longrightarrow L^2(\R)$ is the linearization about the critical front, defined through
\begin{equation}
\label{e: unweighted linearized EFKPP}
\Cal{A}(\delta) := -\delta^2 \partial_x^4 + \partial_x^2 + c_* \partial_x + f'(q_*(x; \delta)).
\end{equation} 
The essential spectrum of the linearization $\mathcal{A}(\delta)$ is unstable due to the instability of the background state $u \equiv 0$. Hence, to establish a stability result, one restricts to perturbations with prescribed exponential localization. The optimal exponential weight here matches the decay rate of the critical front; we therefore define
\begin{align}
\omega_* (x; \delta) = \begin{cases}
e^{\eta_*(\delta) x}, & x \geq 1, \\
1, & x \leq -1,
\end{cases}
\label{e: critical weight},
\end{align}
so that the conjugate operator $\mathcal{L} (\delta) = \omega_* (\cdot; \delta) A(\delta) \omega_*(\cdot; \delta^{-1}) : H^4 (\R) \to L^2 (\R)$ describes the linearized dynamics of perturbations in this weighted space. The essential spectrum of $\mathcal{L}(\delta)$ is marginally stable, touching the imaginary axis only at the origin; see Figure \ref{f: spectrum A/L} and Lemma \ref{l: pinched double root} for details. Our main result establishes spectral stability for $\mathcal{L}(\delta)$ as required by the marginal stability conjecture in light of \cite{AveryScheelSelection}. 
\begin{thm}[Spectral stability] \label{t: spectral stability}
There exists a $\delta_0 > 0$ such that for all $\delta \in (- \delta_0, \delta_0)$ the operator $\mathcal{L}(\delta)$ has no eigenvalues $\lambda$ with $\mathrm{Re } \, \lambda \geq 0$, and there does not exist a bounded solution to $\mathcal{L}(\delta) u = 0$. 
\end{thm}
Together with Lemmas \ref{l: pinched double root} and \ref{l: left stability} which control the essential spectrum, Theorem \ref{t: spectral stability} says that the critical front is marginally spectrally stable. The results in \cite{AveryScheel} therefore imply nonlinear stability of the critical front against localized perturbations, with sharp decay rates and precise characterization of the leading order asymptotics. To state these, we first define for $r \in \R$ a smooth positive one-sided algebraic weight $\rho_r$ which satisfies
\begin{align*}
\rho_r (x) = 
\begin{cases}
1, & x \leq -1, \\
(1+x^2)^{r/2}, & x \geq 1. 
\end{cases}
\end{align*}
We then have the following nonlinear stability results.

\begin{corollary}[Nonlinear stability]\label{c: nonlinear stability}
Let $r > \frac{3}{2}$. There exist constants $\eps > 0$ and $C > 0$ such that if $\|\omega_{*} \rho_r v_0\|_{H^1} < \eps$, then 
\begin{align*}
\| \omega_* \rho_{-r} v(t, \cdot)\|_{H^1} \leq \frac{C \eps}{(1+t)^{3/2}},
\end{align*}
where $v$ is the solution to \eqref{e: perturbation eqn} with initial data $v_0$. Furthermore if $r > \frac{5}{2}$, then there exists a real number $\alpha_* = \alpha_*(\omega_{*} \rho_r v_0)$, depending smoothly on $\omega_{*} \rho_r v_0$ in $H^1 (\R)$, such that for $t > 1$,
\begin{align*}
\| \rho_{-r} \omega_* (v(t, \cdot) - \alpha_* t^{-3/2} q_*'(\cdot; \delta))\|_{H^1} \leq \frac{C \eps}{(1+t)^2}.
\end{align*}
\end{corollary}

Nonlinear stability of the critical front in the classical Fisher-KPP equation, $\delta = 0$, against localized perturbations was established by Kirchg\"assner \cite{Kirchgassner} and later refined in \cite{EckmannWayne, Gallay, FayeHolzer, AveryScheel}. The sharp $t^{-3/2}$ decay rate in this setting was first established in \cite{Gallay} and later reobtained in \cite{FayeHolzer, AveryScheel}. Crucial to this improved decay compared to the standard diffusive decay rate $t^{-1/2}$ is the lack of an embedded eigenvalue of the linearization at $\lambda = 0$, as captured here in Theorem \ref{t: spectral stability}, an observation made precise in \cite{AveryScheel}. Nonlinear stability of the critical front for $\delta \neq 0$ was obtained in \cite{RottschaferWayne} via weighted energy estimates, but without a precise characterization of the decay rate, while the $t^{-3/2}$ decay rate obtained here is sharp in light of the asymptotics given in Corollary \ref{c: nonlinear stability}. 

Finally, the spectral stability obtained in Theorem \ref{t: spectral stability}, Lemma \ref{l: pinched double root}, and Lemma \ref{l: left stability}, together with the analysis in \cite{AveryScheelSelection} confirms the marginal stability conjecture for \eqref{e: EFKPP}. 

\begin{corollary}[Front selection]\label{c: front selection}
Fix $r > 2$. For any $\eps > 0$ there exists a class of initial data $\mathcal{U}_\eps$, including nontrivial data supported on a half-line, such that for any $u_0 \in \mathcal{U}_\eps$, we have 
\begin{align*}
\sup_{x \in \R} | \rho_{-1} (x) \omega_* (x; \delta) [u(x + \sigma(t), t) - q_*(x; \delta)] | < \eps,
\end{align*}
where $u$ is the solution to \eqref{e: EFKPP} with initial data $u_0$, and 
\begin{align*}
\sigma(t) = c_* (\delta) t - \frac{3}{2 \eta_*(\delta)} \log t + x_\infty (u_0)
\end{align*}
for some $x_\infty (u_0) \in \R$. Moreover, $\mathcal{U}_\eps$ is open in the topology induced by the norm $\| f \| = \| \rho_r \omega_* f \|_{L^\infty}$. 
\end{corollary}
This result confirms that open classes of steep initial data propagate with the linear spreading speed $c_* (\delta)$, up to a universal logarithmic delay, as predicted by the marginal stability conjecture \cite{vanSaarloos}; see \cite{AveryScheelSelection} for further details. In the classical Fisher-KPP equation, $\delta = 0$, analogous convergence results for non-negative steep data may be shown using comparison principles \cite{Aronson, Comparison1, Comparison2, Lau} or probabilistic methods \cite{Bramson1, Bramson2}. We believe Corollary \ref{c: front selection} represents an important step in extending results on front selection beyond equations with comparison principles and toward pattern forming systems. 

\subsection{Remarks}\label{s: remarks}
\noindent \textbf{Assumptions on $f$.} Since we prove our results by perturbing from the classical Fisher-KPP equation, our results hold for any smooth nonlinearity $f$ which satisfies $f(0) = f(1) = 0$, $f'(0) > 0, f'(1) < 0$, and for which existence and spectral stability of the critical front hold for the classical Fisher-KPP equation with this reaction term. In particular, this is implied by the assumption $f''(u) < 0$ for $u \in (0,1)$ \cite[Theorem 5.5]{Sattinger}, which we state in the introduction. This can be weakened, for instance, to the assumption that $0 < f(u) \leq f'(0) u$ for $u \in (0,1)$; see e.g. \cite{Aronson}. 

\noindent \textbf{General approach -- preconditioning.} Our approach to regularizing the singular perturbation is based on preconditioning with an appropriately chosen operator. To illustrate the main idea, briefly consider the eigenvalue problem for the unweighted linearization, $(\mathcal{A}(\delta) - \lambda) u = 0$. Applying $(1- \delta^2 \partial_x^2)^{-1}$ to $\mathcal{A}(\delta) - \lambda$, we obtain 
\begin{align*}
(1-\delta^2 \partial_x^2)^{-1} (\mathcal{A} (\delta) - \lambda) &= (1-\delta^2 \partial_x^2)^{-1} [(1-\delta^2 \partial_x^2) \partial_x^2 + c_*(\delta) \partial_x + f'(q_*) - \lambda] \\
&= \partial_x^2 + (1-\delta^2 \partial_x^2)^{-1} (c_*(\delta) \partial_x + f'(q_*) - \lambda) \\
&= \partial_x^2 + c_*(\delta) \partial_x + f'(q_* (\cdot; \delta)) - \lambda + T(\delta) \big(c_*(\delta) \partial_x + f'(q_*(\cdot; \delta)) - \lambda \big),
\end{align*}
where $T(\delta) = (1-\delta^2 \partial_x^2)^{-1} - 1$. 
Once we prove that the terms involving $T(\delta)$ are continuous in $\delta$, the eigenvalue problem becomes essentially a regular perturbation of the
classical Fisher-KPP linearization, at $\delta = 0$. We prove the necessary estimates on the preconditioners using direct Fourier analysis in Section \ref{s: preconditioner estimates}. This approach is inspired by that used to construct oblique stripe solutions in a quenched Swift-Hohenberg equation in \cite{GohScheel}. 

\noindent \textbf{Stability to less localized perturbations.} We note that under the spectral stability conditions we prove here, in addition to Corollary \ref{c: nonlinear stability}, one also immediately obtains from the results of \cite{AveryScheel} stability under less localized perturbations, with a prescribed decay rate which is slower than $t^{-3/2}$. See \cite[Theorems 3 and 4]{AveryScheel} for details. 

\noindent \textbf{Geometric vs. functional analytic point of view.} We remark here that one should also be able to prove the spectral stability results obtained here using geometric dynamical systems methods, in particular geometric singular perturbation theory in the sense of Fenichel \cite{Fenichel} together with the gap lemma \cite{GardnerZumbrun, KapitulaSandstede}, which is used to extend the Evans function into the essential spectrum. An attractive feature of our approach here is that it is quite self contained, ultimately relying mostly on basic Fredholm theory and Fourier analysis. We also remark that in principle the functional analytic methods could be adapted, together with the approach to linear stability through obtaining resolvent estimates via far-field/core decompositions in \cite{AveryScheel}, to problems in stability of critical fronts in nonlocal equations, since these methods do not rely as heavily on the presence of an underlying phase space. Some of the relevant Fredholm theory for nonlocal operators has been developed in \cite{FayeScheel1, FayeScheel2}. 

\noindent \textbf{Natural range for $\delta$.}
In this paper, we have restricted to small $\delta$. However, we believe that similar results should hold true for larger values of this parameter. While the existence of fronts is established in \cite{Berg_Hulshof_Vandervorst_01} for all speeds $c>0$ and $\delta\in \R$, we do not have access to explicit decay at $+\infty$ for this fronts, which seems necessary to establish precise stability. Monotonicity of the front would imply such a precise decay by use of Ikehara's theorem \cite{CarrChmaj}. An important value is $\bar{\delta} = 1/\sqrt{12 f'(0)}$, at which the dispersion relation admits a triple root, and the essential spectrum of the linearized operator becomes tangent to the imaginary axis. Stability at or above this value of $\delta$ is therefore fundamentally outside the scope of \cite{AveryScheel}. 

\noindent \textbf{Supercritical and subcritical fronts.} If we consider a supercritical front, traveling with speed $c > c_*(\delta)$ and constructed in \cite{RottschaferWayne}, one can simplify the argument of Theorem \ref{t: spectral stability} to prove that the linearization about such a front has no unstable point spectrum. For these fronts, one can use an exponential weight to push the essential spectrum entirely into the left half plane, and thereby with the analogue of Theorem \ref{t: spectral stability} obtain stability of supercritical fronts with an exponential decay rate using standard semigroup methods (see e.g. \cite{Henry}). Subcritical fronts, with $c < c_*(\delta)$, have unstable absolute spectrum, meaning in particular that the essential spectrum of the linearization about any of these fronts is unstable in any exponentially weighted space. A modified version of our proof of Theorem \ref{t: existence} should also give existence of these supercritical and subcritical fronts using functional analytic methods, although we do not give the details here. 

\noindent \textbf{Additional notation.} For $r > 0$, we let $B(0, r)$ denote the ball of radius $r$ centered at the origin in the complex plane.  

\noindent \textbf{Outline.} The remainder of this paper is organized as follows. In Section \ref{s: preliminaries}, we compute some preliminary information needed for our analysis (the linear spreading speed in \eqref{e: EFKPP fronts} and the cokernel of $\mathcal{L}(0)$) and prove some necessary estimates on our preconditioner. In Section \ref{s: existence}, we use explicit preconditioners and a far-field/core decomposition to prove Theorem \ref{t: existence}, establishing existence of the critical front. In Section \ref{s: small eigenvalues}, we define a functional analytic analogue of the Evans function near $\lambda = 0$, and use it together with knowledge of the spectrum of $\mathcal{L}(0)$ to prove that $\mathcal{L}(\delta)$ has no resonance at the origin or unstable eigenvalues for $\delta$ small. In Section \ref{s: large eigenvalues}, we complete the proof of Theorem \ref{t: spectral stability} by showing that there are also no unstable eigenvalues away from the origin. 

\noindent \textbf{Acknowledgements.} The authors are grateful to Arnd Scheel and Gr\'egory Faye for helpful comments. MA was supported by the National Science Foundation through the Graduate Research Fellowship Program under Grant No. 00074041. Any opinions,
findings, and conclusions or recommendations expressed in this material are those of the
authors and do not necessarily reflect the views of the National Science Foundation.

\section{Preliminaries}\label{s: preliminaries}

\subsection{Exponential weights}
\label{s: exponential weights}
In addition to the critical weight \eqref{e: critical weight} which we use to shift the essential spectrum out of the right half plane, we will need further exponential weights to recover Fredholm properties of $\mcl(\delta)$ and related operators for our far-field/core analysis. For $\eta_{\pm} \in \R$, we define a smooth positive weight function $\omega_{\eta_-, \eta_+}$ satisfying
\begin{align*}
\omega_{\eta_-, \eta_+} = \begin{cases}
e^{\eta_- x}, &x \leq -1, \\
e^{\eta_+ x}, &x \geq 1. 
\end{cases}
\end{align*}
If $\eta_- = 0$ and $\eta_+= \eta$, then we write $\omega_{\eta_-, \eta_+} = \omega_\eta$. If $\eta_- = \eta_+ = \eta$, we choose $\omega_{\eta, \eta} (x) = e^{\eta x}$.

Given an integer $m$, we define the exponentially weighted Sobolev space $H^m_{\eta_-, \eta_+} (\R)$ through the norm 
\begin{align*}
||f||_{H^m_{\eta_-, \eta_+}} = || \omega_{\eta_-, \eta_+} f||_{H^m}. 
\end{align*}
We note that for $\eta> 0$ we have $H^m_{0, \eta} (\R) = H^m (\R) \cap H^m_{\eta, \eta} (\R)$ as well as the following equivalence of norms
\begin{align}
||f||_{H^m_{0, \eta}} \sim ||f||_{H^m} + ||f||_{H^m_{\eta, \eta}}. \label{e: weighted norms equivalent}
\end{align}
This characterization of the one-sided weighted spaces is useful in obtaining estimates on operators defined by Fourier multipliers on these spaces, and we make use of this below in Section \ref{s: preconditioner estimates}. 

\subsection{Linear spreading speed and essential spectrum}
The linear spreading speed, marking the transition from pointwise growth to pointwise decay in the linearization about $u \equiv 0$, is characterized here by the location of simple pinched double roots of the \emph{dispersion relation}
\begin{equation}
\label{e: dispersion relation}
d_c^+(\lambda, \nu) = - \delta^2 \nu^4 + \nu^2 + c \nu + f'(0) - \lambda;
\end{equation}
see \cite{HolzerScheelPointwiseGrowth} for background. 

\begin{lemma}[Linear spreading speed]\label{l: pinched double root}
There exists $\delta_0> 0$ such that for $\delta\in(-\delta_0, \delta_0)$, there exists a critical speed $c_* = c_*(\delta)$, and an exponent $\eta = \eta_*(\delta) > 0$ for the critical weight such that the right dispersion relation \eqref{e: dispersion relation} satisfies the following properties.
\begin{enumerate}[label=$\mathit{(\roman*)}$]
\item \label{h: 1-1} Simple pinched double root: for $\lambda$, $\nu$ near $0\in\C$:
\begin{equation}
d_{c_*}^+(\lambda, -\eta_* + \nu) = \nu^2 \sqrt{1 - 12 \delta^2 f'(0)} - \lambda + \Rm{O}(\nu^3), 
\label{e: pinched double root}
\end{equation}
with $\sqrt{1 - 12 \delta^2 f'(0)} > 0$.
\item \label{h: 1-2} Minimal critical spectrum: if $d^+_{c_*} (i \kappa, -\eta_*+ik) = 0$ for some $\kappa, k \in \R$, then $\kappa = k = 0$. 
\item \label{h: 1-4} No unstable essential spectrum: if $d^+_{c_*}(\lambda, -\eta_* + ik) = 0$ for some $k \in \R$ and $\lambda \in \C$, then $\real{\lambda} \leq 0$. 
\end{enumerate}
\end{lemma}
We prove Lemma \ref{l: pinched double root} below, but first we explain how this lemma determines the essential spectrum for $\mathcal{L}(\delta)$, the linearization about the critical front in the exponentially weighted space with critical weight determined by this lemma. The operator $\mathcal{L}(\delta)$ has the precise form
\begin{equation*}
\label{e: operator L}
\Cal{L}(\delta) = \omega_* \Cal{A}(\delta) \omega_*^{-1} = - \delta^2 \partial_x^4 + \delta^2 a_3 \partial_x^3 + \left(1 + \delta^2 a_2\right) \partial_x^2 + a_1\partial_x + a_0,
\end{equation*}
where the coefficients $a_i(x;\delta)$ converge to limits $a_i^\pm(\delta)$ exponentially quickly when $x\to \pm \infty$, and are defined using the local notation $\varpi(x) := 1/\omega_*(x)$ by the following expressions:
\begin{equation}
\label{e: coefficients a}
a_3 = -4\frac{\varpi'}{\varpi},
\hspace{2em}
a_2 =-6\frac{\varpi^{''}}{\varpi},
\hspace{2em}
a_1 = c_* + 2\frac{\varpi'}{\varpi} - 4\delta^2\frac{\varpi^{'''}}{\varpi},
\hspace{2em}
a_0 = f'(q_*) + c_*\frac{\varpi'}{\varpi}+ \frac{\varpi^{''}}{\varpi} - \delta^2\frac{\varpi^{''''}}{\varpi}.
\end{equation}
We note that $\varpi^{(k)}(x)/\varpi(x) = (-\eta_*)^k$ for $x\geq 1$.

\begin{figure}
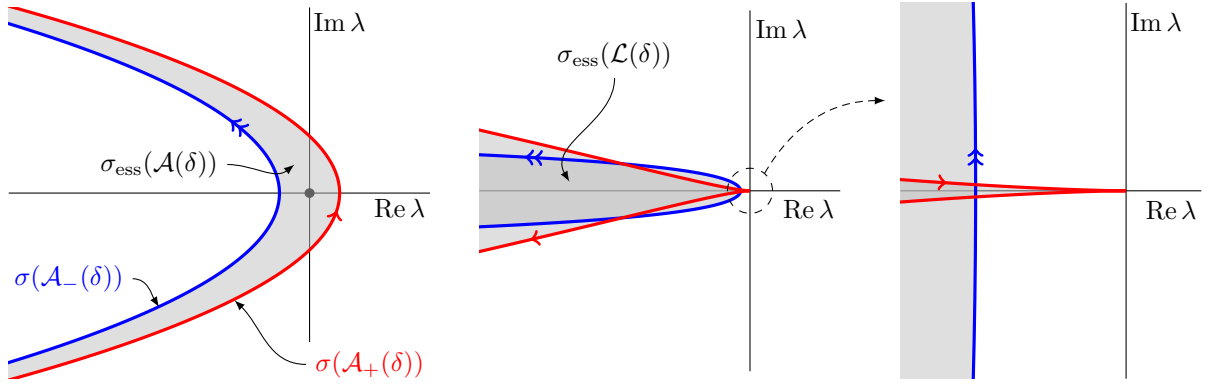

\hspace{\stretch{1}}
\includegraphics[scale=\figurescale]{\ImagePath{spectrum_A}}
\hspace{\stretch{2}}
\includegraphics[scale=\figurescale]{\ImagePath{spectrum_L}}
\hspace{\stretch{1}}
\caption{Left: the essential spectrum of the unweighted operator $\Cal{A}(\delta)$ is bounded by the Fredholm borders in blue and red. Middle and right: overview and zoom near the origin of the essential spectrum of the weighted operator $\Cal{L}(\delta)$.}
\label{f: spectrum A/L}
\end{figure}

For such a linear operator, the essential spectrum is delimited by the two Fredholm borders, which are defined using the asymptotic dispersion relations. More precisely, the boundaries of the essential spectrum of $\Cal{L}$ are determined by the essential spectrum of the limiting operators $\Cal{L}_\pm$, obtained by sending $x\to \pm \infty$ \cite{KapitulaPromislow, FiedlerScheel}. From the construction of $c_*$, $\eta_*$ (see the proof of Lemma \ref{l: pinched double root} below), we have at $+\infty$:
\begin{equation}
\label{e: operator L+}
\Cal{L}_+(\delta) = -\delta^2 \partial_x^4 + 4 \eta_* \delta^2 \partial_x^3 +(1- 6\delta^2 \eta_*^2) \partial_x^2.
\end{equation}
The spectrum of this constant coefficient operator is, via Fourier transform, readily seen to be marginally stable; see the red curves of Figure \ref{f: spectrum A/L}. Notice that for $\delta$ small, $\eta=\eta_*$  is the only reasonable value for which $\Cal{L}_+$ has a non positive zeroth order term; any other choice of $\eta(\delta)$ will lead to spectral instability for $\Cal{L}(\delta)$.
At $-\infty$, there is no contribution from $\omega_*$, hence
\begin{equation*}
\Cal{L}_- = \Cal{A}_- = -\delta^2 \partial_x^4 + \partial_x^2 + c_* \partial_x + f'(1)
\end{equation*}
has a stable spectrum, with spectral gap $f'(1)<0$. Via the Fourier transform, this spectrum is determined by the asymptotic dispersion relation
\begin{equation*}
d_{c_*}^- (\lambda, \nu) = - \delta^2 \nu^4 + \nu^2 + c_* (\delta) \nu + f'(1) - \lambda. 
\end{equation*}
\begin{lemma}[Stability on the left]\label{l: left stability}
If $d^-_{c_*} (\lambda, ik) = 0$ for some $k \in \R$, then $\real{\lambda} < 0$. 
\end{lemma}
Lemmas \ref{l: pinched double root} and \ref{l: left stability} together with Palmer's theorem \cite{Palmer1, Palmer2} imply that the essential spectrum of $\mathcal{L}(\delta)$ is marginally stable, touching the imaginary axis only at the origin \cite{KapitulaPromislow, FiedlerScheel}; see Figure \ref{f: spectrum A/L}.

\begin{proof}[Proof of Lemma \ref{l: pinched double root}]
We first look for $c_*, \eta_* > 0$ which satisfy \eqref{e: pinched double root}. The polynomial $\nu \mapsto d_c(\lambda,\nu)$ at $\lambda = 0$ admits $-\eta$ as a double root if and only if
\begin{equation*}
\label{e: double root system}
\left\lbrace
\begin{array}{l}
0 = d_c^+(0, -\eta) = -\delta^2\,\eta^4 + \eta^2 - c \,\eta + f'(0),\\
0 = \partial_\nu d_c^+(0, -\eta) = 4\delta^2\, \eta^3 - 2\,\eta + c.
\end{array}
\right.
\end{equation*}
We remove $c$ from the first equation by using the second one, and find a quadratic equation satisfied by $\eta^2$, which has roots $\pm \eta_1, \pm \eta_2$ where 
\begin{equation}
\label{e: definition eta}
\eta_1 := \frac{1}{\absolu{\delta}\sqrt{6}} \sqrt{1 + \sqrt{1-12\delta^2 f'(0)}} \sim \frac{1}{\absolu{\delta}\sqrt{3}},
\hspace{4em}
\eta_2 := \frac{1}{\absolu{\delta}\sqrt{6}} \sqrt{1 - \sqrt{1-12\delta^2 f'(0)}} \sim \sqrt{f'(0)}.
\end{equation}
where the asymptotics hold for $\delta \to 0$.
The choice $\eta_* =\eta_2$ and $c_* = 2\eta_* - 4 \delta^2 \eta_*^3$ leads to
\begin{equation*}
c_*(\delta) = 2\sqrt{f'(0)} - \delta^2 f'(0)^{3/2} +\Rm{O}(\delta^4).
\end{equation*}
The other double roots do not determine linear spreading speeds, as they are not \textit{pinched}; see \cite{HolzerScheelPointwiseGrowth} for details. We now fix $\delta_0 = 1/\sqrt{12f'(0)}$. Then for $\absolu{\delta} < \delta_0$, and using the expression of $\eta_*= \eta_2$, we obtain:
\begin{equation*}
\frac{\partial_\nu^2 d_c^+(0, -\eta_*)}{2!} = 1 - 6 \delta^2 \eta_*^2 = \sqrt{1 - 12 \delta^2 f'(0)} > 0.
\end{equation*}
Hence $(\lambda, \nu) = (0, -\eta_*)$ is a simple double root and \eqref{e: pinched double root} is proved. Such an expansion together with the lack of unstable essential spectrum ensures that this root is pinched; see \cite[Lemma 4.4]{HolzerScheelPointwiseGrowth}. Alternatively, Lemma \ref{l: ansatz} below directly proves that the root is pinched.

We now check the two remaining conditions in Lemma \ref{l: pinched double root}. We equate the polynomial $P(X) = d_*^+(\lambda,X)$ with its Taylor series centered at the double root $X = -\eta_*$ to obtain
\begin{equation}
\label{e: double root 1}
\real{d_*^+(\lambda, -\eta_* + ik)} = - \real \lambda + \real\bigg(\sum_{j=0}^4 (ik)^j\frac{P^{(j)}(-\eta_*)}{j!}  \bigg) = - \real \lambda -\delta^2 k^4 - (1 - 6\delta^2 \eta_*^2) k^2 \leq 0
\end{equation}
if $\real \lambda > 0$, since from \eqref{e: definition eta}, we have $1 - 6\delta^2 \eta_*^2 = \sqrt{1 - 12 \delta^2 f'(0)}\geq 0$. 
This proves hypothesis \ref{h: 1-4}. Furthermore, the inequality in \eqref{e: double root 1} is an equality if and only if $k=0$ and $\lambda = 0$, for which we have $d_*^+(0, -\eta_*) = 0$. Hence, hypothesis \ref{h: 1-2} is proved.
\end{proof}

\subsection{Preconditioner estimates} \label{s: preconditioner estimates}

Here we prove the estimates we will need on our preconditioner $(1-\delta^2 \partial_x^2)^{-1}$, by directly examining its Fourier symbol. 
\begin{lemma}\label{l: preconditioner estimates}
Fix $\eta > 0$ sufficiently small, and fix an integer $m$. Then there exist constants $\delta_0 > 0$ and $C = C(\delta_0, \eta)$ such that if $|\delta| < \delta_0$, 
\begin{align}
||(1- \delta^2 \partial_x^2)^{-1}||_{L^2_{0, \eta} \to L^2_{0, \eta}} &\leq C, \label{e: precond L2 L2 bound} \\
||(1- \delta^2 \partial_x^2)^{-1} ||_{H^m_{0, \eta} \to H^{m+1}_{0, \eta}} &\leq \frac{C}{|\delta|}. \label{e: precond Hm Hmplus1 bound}
\end{align}
\end{lemma}

\begin{proof}
By \eqref{e: weighted norms equivalent}, it suffices to prove the estimates separately for $L^2$ and for $L^2_{\eta, \eta}$ with $\eta > 0$ small. 
%
%
Since multiplication by $e^{\eta \cdot}$ is an isomorphism from $L^2_{\eta, \eta} (\R)$ to $L^2 (\R)$, to prove estimates for $(1-\delta^2 \partial_x^2)^{-1}$ on $L^2_{\eta, \eta}$, it suffices to consider the inverse of the conjugate operator 
\begin{align*}
e^{\eta \cdot} (1-\delta^2 \partial_x^2) e^{-\eta \cdot} = 1 - \delta^2 (\partial_x-\eta)^2
\end{align*}
acting on $L^2(\R)$. This is the advantage of using \eqref{e: weighted norms equivalent} to separate estimates on $L^2_{0, \eta} (\R)$ into estimates on $L^2 (\R)$ and $L^2_{\eta, \eta} (\R)$: the conjugate operator arising from studying $(1- \delta^2 \partial_x^2)$ on $L^2_{\eta, \eta} (\R)$ has constant coefficients since the weight is a fixed exponential function, and so we can directly estimate its inverse using the Fourier transform. 

Fix $\eta \geq 0$. By Plancherel's theorem, 
\begin{align*}
\| (1-\delta^2 (\partial_x - \eta)^2)^{-1} f||_{L^2} = \left\| \frac{1}{1 - \delta^2 (i \cdot - \eta)^2} \hat{f} (\cdot) \right\|_{L^2} \leq \sup_{k \in \R} \left| \frac{1}{1 - \delta^2 (ik - \eta)^2}  \right| ||\hat{f}||_{L^2}. 
\end{align*}
Let $\delta_0 = 1/(\sqrt{2} \eta)$, so that $\delta_0^2 \eta^2 = 1/2$, and hence if $|\delta| < \delta_0$,
\begin{align}
1 + \delta^2(k^2 - \eta^2) = 1 - \delta^2 \eta^2 + \delta^2 k^2 \geq \frac{1}{2} + \delta^2 k^2. \label{e: precond symbol estimate}
\end{align}
Then for any $\delta$ with $|\delta| < \delta_0$, we have
\begin{align*}
\left| \frac{1}{1 - \delta^2 (ik - \eta)^2}  \right|^2 = \frac{1}{(1+\delta^2(k^2 - \eta^2))^2 + 4 k^2 \delta^4 \eta^2 } \leq \frac{1}{(1+\delta^2(k^2 - \eta^2))^2} \leq \frac{1}{\frac{1}{2} + \delta^2 k^2} \leq C,
\end{align*}
with $C$ depending only on $\delta_0$ and $\eta$, and so 
\begin{align*}
\|(1-\delta^2 \partial_x^2)^{-1}\|_{L^2_{\eta, \eta} \to L^2_{\eta, \eta}} \leq C.
\end{align*}
Since this holds for any fixed $0 \leq \eta < 1$, in particular also for $\eta = 0$, we obtain \eqref{e: precond L2 L2 bound} by combining these estimates with \eqref{e: weighted norms equivalent}. 

Now we prove \eqref{e: precond Hm Hmplus1 bound}, again by obtaining bounds on the Fourier symbol of the inverse of the conjugate operator for $\eta > 0$ and $\eta = 0$. By Plancherel's theorem, for any fixed $0 \leq \eta < 1$, we have 
\begin{align*}
||(1-\delta^2 (\partial_x^2-\eta))^{-1} f||_{H^{m+1}} = \left\| \frac{1}{1 - \delta^2 (i \cdot - \eta)^2} \langle \cdot \rangle^{m+1} \hat{f} (\cdot) \right\| \leq \sup_{k \in \R} \left| \frac{1}{1- \delta^2(ik-\eta)^2} \langle k \rangle \right| \| \hat{f} \|_{H^m}. 
\end{align*}
Again, let $\delta_0 = 1/(\sqrt{2} \eta)$. Then, by \eqref{e: precond symbol estimate}, we have 
\begin{align*}
\left| \frac{1}{1- \delta^2(ik-\eta)^2} \langle k \rangle \right|^2 = \frac{\delta^2 +\delta^2 k^2}{(1+\delta^2(k^2-\eta^2))^2 + 4 k^2 \delta^4 \eta^2} \frac{1}{\delta^2} \leq \frac{\delta^2+\delta^2 k^2}{\left( \frac{1}{2} + \delta^2 k^2 \right)^2} \frac{1}{\delta^2} \leq \frac{C}{\delta^2}, 
\end{align*}
from which we obtain 
\begin{align*}
||(1-\delta^2 (\partial_x^2 - \eta))^{-1} f||_{H^{m+1}} \leq \frac{C}{|\delta|} ||f||_{L^2}. 
\end{align*}
Since this holds for $\eta \geq 0$, we obtain \eqref{e: precond Hm Hmplus1 bound} from the equivalence of norms \eqref{e: weighted norms equivalent}. 
\end{proof}

We now state and prove the estimates we will need on the difference between the preconditioner and the identity, $T(\delta) = (1-\delta^2 \partial_x^2)^{-1} -1$.
\begin{lemma} \label{l: T delta estimates}
Fix $\eta > 0$ sufficiently small. There exists a constant $\delta_0$ such that the mapping $\delta \mapsto T(\delta)$ is continuous from $(-\delta_0, \delta_0)$ to $\mathcal{B}(H^1_{0, \eta}, L^2_{0, \eta})$, the space of bounded linear operators from $H^1_{0, \eta} (\R)$ to $L^2_{0, \eta} (\R)$ with the operator norm topology. 
\end{lemma}
\begin{proof}
As in the proof of Lemma \ref{l: preconditioner estimates}, it suffices to establish continuity in $\delta$ of the conjugate operator $T_\eta (\delta) := (1-\delta^2(\partial_x - \eta)^2)^{-1} - 1$ on $L^2(\R)$ for $\eta \geq 0$ sufficiently small. For $\delta$ nonzero, we write 
\begin{align*}
1-\delta^2 (\partial_x - \eta)^2 = \delta^2 \left( \frac{1}{\delta^2} - (\partial_x - \eta)^2 \right). 
\end{align*} 
By standard spectral theory, we therefore see that $T_\eta(\delta)$ is continuous in $\delta$ provided $\delta$ is nonzero and $1/\delta^2$ is in the resolvent set of the operator $(\partial_x - \eta)^2$. Computing the spectrum of this operator with the Fourier transform, one readily finds that there exists a $\delta_1$ depending on $\eta$ such that this continuity holds for $0 < \delta < \delta_1$. 

We now establish continuity at $\delta = 0$ via direct estimates on the Fourier multiplier. 
\begin{align*}
\hat{T}_\eta(\delta, k) = \frac{\delta^2 (ik - \eta)^2}{1 - \delta^2 (ik - \eta)^2}.
\end{align*}
Since we are proving continuity of $T_{\eta} (\delta)$ from $H^1$ to $L^2$, we gain a helpful factor of $\langle k \rangle$ --- that is, it suffices to estimate $|\hat{T}_\eta(\delta, k)|/\langle k \rangle$. By \eqref{e: precond symbol estimate}, for $|\delta| < \delta_0 := \min \{ \delta_1, 1/\sqrt{2} \eta \}$ we have
\begin{align*}
\left| \hat{T}_{\eta} (\delta, k) \frac{1}{\langle k \rangle} \right| = \left| \frac{\delta^2 (ik - \eta)^2}{1 - \delta^2 (ik - \eta)^2} \frac{1}{\langle k \rangle} \right| &= \frac{\delta^2 \sqrt{(\eta^2 - k^2)^2+4k^2 \eta^2}}{\sqrt{(1-\delta^2(\eta^2-k^2))^2 + 4 \delta^4 k^2 \eta^2}} \frac{|\delta|}{(\delta^2 + \delta^2 k^2)^{1/2}} \\
&\leq |\delta| \left( \frac{\delta^4 \eta^4 + \delta^4 k^4 + 2 \delta^4 k^2 \eta^2}{\left[\left( \frac{1}{2} + \delta^2 k^2 \right)^2 + 4 \delta^4 k^2 \eta^2 \right] (\delta^2 + \delta^2 k^2)} \right)^{1/2}, 
\end{align*} 
using \eqref{e: precond symbol estimate} in the denominator. We now split the factor in the parenthesis, first estimating 
\begin{align*}
\frac{\delta^4 \eta^4 + 2 \delta^4 k^2 \eta^2}{\left[\left( \frac{1}{2} + \delta^2 k^2 \right)^2 + 4 \delta^4 k^2 \eta^2 \right] (\delta^2 + \delta^2 k^2)} \leq \frac{\delta^4 \eta^4 + 2 \delta^4 k^2 \eta^2}{\frac{1}{4} (\delta^2 + \delta^2 k^2)} = \frac{\delta^2 \eta^4 + 2 \delta^2 k^2 \eta^2}{\frac{1}{4} (1+k^2)} \leq C,
\end{align*}
where $C$ depends only on $\delta_0$ and $\eta$. For the remaining term, we have 
\begin{align*}
\frac{\delta^4 k^4}{\left[\left( \frac{1}{2} + \delta^2 k^2 \right)^2 + 4 \delta^4 k^2 \eta^2 \right] (\delta^2 + \delta^2 k^2)} \leq \frac{\delta^4 k^4}{\left( \frac{1}{2} + \delta^2 k^2 \right)^2 (\delta^2 k^2)} = \frac{\delta^2 k^2}{\left(\frac{1}{2} + \delta^2 k^2 \right)^2} \leq C,
\end{align*}
again with constant $C$ only depending on $\eta$ and $\delta_0$. From this estimate on the Fourier symbol together with Plancherel's theorem, we obtain
\begin{align*}
||T_\eta(\delta)||_{H^1 \to L^2} \leq C |\delta|, 
\end{align*}
for $|\delta| < \delta_0$, and so in particular $\delta \mapsto T_\eta(\delta)$ is continuous at $\delta = 0$, which completes the proof of the lemma. 
\end{proof}

\subsection{Fredholm properties at $\delta = 0$}
\label{s: Fred Prop L0}
We will further need the Fredholm properties of $\Cal{L}(0)$, which is the linearization in the weighted space of the classical FKPP problem $\delta=0$. The classical Fisher-KPP front, at $\delta = 0$, may be constructed via simple phase plane methods (see \cite{Sattinger}), and we denote this front by $q_0$. In the following two lemmas, we describe the kernel, the cokernel and the range of $\Cal{L}(0)$. They will both be needed for the existence of the critical front $q_*(\cdot;\delta)$ in Section \ref{s: existence}, and for the control of small eigenvalues in Section \ref{s: small eigenvalues}.

\begin{lemma}\label{l: L0 fredholm properties}
For $\eta > 0$, the operator $\mathcal{L}(0) : H^2_{0, \eta} (\R) \to L^2_{0, \eta} (\R)$ is Fredholm with index $-1$, with trivial kernel and with cokernel spanned by $\varphi(x) = (\omega_*(x; 0))^{-1} e^{c_*(0)x} q_0'(x)$. 
\end{lemma}

\begin{figure}
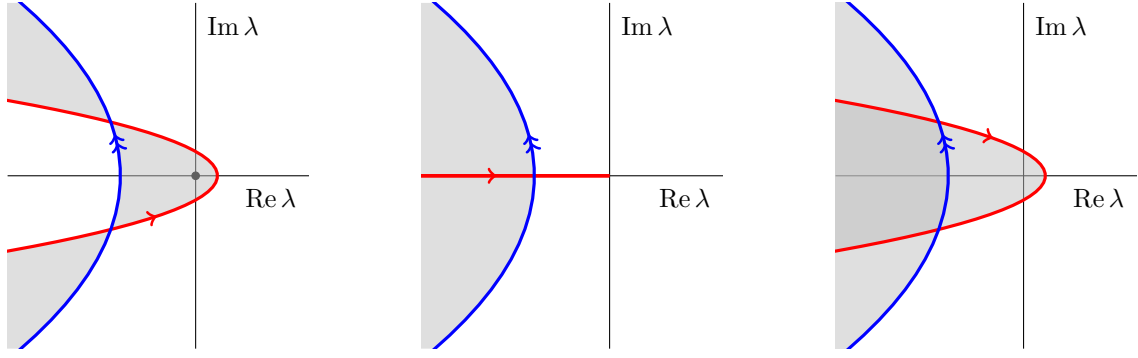

\hspace{\stretch{1}}
\includegraphics[scale=\figurescale]{\ImagePath{spectrum_eta_negative}}
\hspace{\stretch{2}}
\includegraphics[scale=\figurescale]{\ImagePath{spectrum_eta_zero}}
\hspace{\stretch{2}}
\includegraphics[scale=\figurescale]{\ImagePath{spectrum_eta_positive}}
\hspace{\stretch{1}}
\caption{Essential spectrum of $\Cal{L}_\eta(0)$ for $\eta<0$ at left, $\eta=0$ at middle, and $\eta>0$ at right. For $\eta>0$, the positive Fredholm border has reverse orientation so that $\fred(\Cal{L}_\eta) = -1$, while $\fred(\Cal{L}_\eta) = 1$ for $\eta<0$.}
\label{f: spectrum L eta}
\end{figure}

\begin{proof}
Recall that the asymptotic operators are given by $\Cal{L}_+(0) = \partial_x^2$ and $\Cal{L}_-(0) = \partial_x^2 + c_*(0) \partial_x + f'(1)$. For $\eta > 0$, define the conjugate operator:
\begin{equation*}
\Cal{L}_\eta(0) = \omega_{0, \eta} \, \Cal{L}(0) \,  \omega_{0, \eta}^{-1} : H^2(\R) \lra L^2(\R), 
\end{equation*}
with asymptotic operators $\Cal{L}_{\eta, +} = (\partial_x - \eta)^2$ and $\Cal{L}_{\eta, -} = \Cal{L}_-$. Since the multiplication $\omega_{0, \eta} \cdot : L^2_{0, \eta}(\R) \lra L^2(\R)$ is an isomorphism, the Fredholm indexes satisfy 
\begin{equation*}
\fred{\Cal{L}_\eta}(0) = 0 + \fred{\Cal{L}(0)} + 0 = \fred{\Cal{L}(0)}.
\end{equation*}
Then the conjugate operator is defined on a unweighted space, and its Fredholm borders are the two oriented curves $\sigma(\Cal{L}_{\eta,+}) = \lbrace -k^2 + i 2\eta k + \eta^2 : k \in \R\rbrace$ and $\sigma(\Cal{L}_{\eta,-}) = \lbrace -k^2 + i c_*(0) k + f'(1) : k \in \R\rbrace$, which are away from $0\in\C$; see Figure \ref{f: spectrum L eta}. This ensures that $\Cal{L}_\eta(0)$ is Fredholm, we now compute its index $\fred(\Cal{L}_\eta(0) - \lambda)$ at $\lambda = 0$. 
For $\lambda$ to the right of the essential spectrum, we use Palmer's theorem to compute the Fredholm index from the Morse indices (see e.g. \cite{KapitulaPromislow, FiedlerScheel}),
\begin{equation*}
\fred(\Cal{L}_\eta(0)-\lambda) = \dim E_-^\Rm{u}(\lambda) - \dim E_+^\Rm{u}(\lambda) = 2 - 2 = 0.
\end{equation*}
where $E_\pm^\Rm{u}$ are the unstable eigenspaces at $\pm\infty$. To prove that this spaces share the same dimension, one can take $\absolu{\lambda}$ large enough and use a standard normalization; see \cite[proof of Lemma 3.1]{FayeHolzer}. Then the index decreases to $-1$ when $\lambda$ crosses $\sigma(\Cal{L}_{\eta,+})$, since the latter curve has reverse orientation, see \cite{KapitulaPromislow}. Hence at $\lambda = 0$, we have shown that $\fred{\Cal{L}(0)} = \fred{\Cal{L}_\eta}(0) = -1$.

To compute the kernel, we note that 
\begin{equation}
\label{e: kernel L0 1}
u\in \ker{\Cal{L}(0)} \hspace{2em} \text{if and only if} \hspace{2em} u\in H^2_{0, \eta}(\R) \text{ and } \Cal{A}(0) \omega_*^{-1} u = 0,
\end{equation}
with $\Cal{A}(0) = \partial_x^2 + c_*(0)\partial_x + f'(q_0(x))$.
Studying the asymptotic growth of the ODE $\Cal{A}(0)u = 0$, one can construct a basis of solutions $\{ q_0', \phi \}$, with exponential behavior at $-\infty$: $\phi(x)\sim \exp((-\sqrt{f'(0)}-\alpha) x)$ and $q_0'(x) \sim \exp((-\sqrt{f'(0)} + \alpha) x)$ with $\alpha = \sqrt{f'(0) - f'(1)} > \sqrt{f'(0)} > 0$. See \cite[proof of Lemma 2.2]{FayeHolzer} for a similar construction.  Furthermore, the derivative of the front has weak exponential decay at $+\infty$: $q_0'(x) \sim x\omega_*(x)^{-1}$. Hence, neither $\phi$ nor $q_0'$ are sufficiently localized to satisfy the right hand condition in \eqref{e: kernel L0 1}, so that $\ker{\Cal{L}(0)} = \lbrace0 \rbrace$.

Finally, it is easily computed that $\tilde{\Cal{A}}(0) := \exp(\frac{c_*}{2}\cdot) \Cal{A}(0) \exp(-\frac{c_*}{2} \cdot)$ is self-adjoint, so that for $v\in H^2_{0,-\eta}(\R)$ and $u\in H^2_{0, \eta}(\R)$:
\begin{equation*}
\pscal{u, \Cal{L}^* v} = \pscal{\Cal{L} u, v} = \pscal{\Cal{A}(\omega_*^{-1} u), \omega_* v} = \pscal{\tilde{\Cal{A}} (e^{\frac{c_*}{2}\cdot} \omega_*^{-1} u), e^{-\frac{c_*}{2}}\omega_* v} = \pscal{e^{c_* \cdot} \omega_*^{-1} u, \Cal{A}(e^{-c_* \cdot} \omega_* v)},
\end{equation*}
which ensures that $v \in \ker(\Cal{L}(0)^*)$ if and only if $v\in H^2_{0, -\eta}(\R)$ and $\Cal{A}(0) e^{-c_* \cdot}\omega_* v = 0$.
For $x\to -\infty$, $\omega_*(x)^{-1} e^{c_* x} \phi(x) \sim \exp((\sqrt{f'(0)} - \alpha) x)$ is not bounded, hence $\ker(\Cal{L}(0)^*) = \Span(\omega_*^{-1} e^{c_*\cdot} q_0')$. 
\end{proof}

\begin{lemma} \label{l: L0 range}
For $\eta>0$ small enough, the range of $\Cal{L} (0) : H^2_{0, \eta}(\R) \to L^2_{0, \eta}(\R)$ is
\begin{equation*}
\label{e: range L0}
\range(\Cal{L}(0)) = \lbrace u \in L^2_{0,\eta}(\R) : \pscal{u, \varphi} = 0 \rbrace,
\end{equation*}
where $\varphi$ is defined in the above Lemma \ref{l: L0 fredholm properties}. We let $P : L^2_{0,\eta}(\R) \lra \range(\Cal{L}(0))$ denote the orthogonal projection onto $\range{\Cal{L}(0)}$ with respect to the $L^2_{0, \eta}(\R)$-inner product.
\end{lemma}
\begin{proof}
Assume that $u\in \range(\Cal{L}(0))$, so that $u=\Cal{L}(0) \tilde{u}$ with $\tilde{u}\in H^2_{0,\eta}(\R)$. Then $\pscal{u,\varphi} = \pscal{\tilde{u}, \Cal{L}(0)^* \varphi} = 0$. 
To prove the reverse inclusion, write $u\in L^2_{0,\eta}(\R)$ as 
\begin{equation*}
u = P u + (1-P)u.
\end{equation*}
From Lemma \ref{l: L0 fredholm properties}, $\Cal{L}(0)$ is Fredholm, hence its range is closed and $P$ is well defined. Furthermore, $\fred{\Cal{L}(0)} = -1$ and $\ker{\Cal{L}(0)} = \lbrace 0 \rbrace$, so that $1-P$ has a one dimensional range: 
\begin{equation*}
(1-P)u = \alpha(u) \psi,
\end{equation*}
with $\psi\in L^2_{0,\eta}(\R)$ fixed, and $\alpha : H^2_{0,\eta} (\R) \lra \R$ linear.
Assuming that $\pscal{u, \varphi} = 0$, we obtain
\begin{equation}
\label{eqn:LS-closed_range_1}
0 = \pscal{P u, \varphi} + \pscal{(1-P)u,\varphi} = \pscal{\tilde{u}, \Cal{L}(0)^* \varphi} + \alpha(u)\pscal{\psi, \varphi} = \alpha(u)\pscal{\psi,\varphi},
\end{equation}
for some $\tilde{u}\in H^2_\eta(\R)$. Hence either $\alpha(u) =0$ or $\pscal{\psi, \varphi} =0$. If $\pscal{\psi, \varphi} = 0$, then for all $v\in H^2_{0,\eta}(\R)$, we would have $\pscal{v, \varphi} = \pscal{\tilde{v}, \Cal{L}(0)^* \varphi} + \alpha(v) \pscal{\psi,\varphi} = 0$, which is to say that $\varphi = 0$ and is a contradiction. Hence from \eqref{eqn:LS-closed_range_1} we conclude $\alpha(u) = 0$, so that $u = P u \in \range(\Cal{L}(0))$.
\end{proof}

\section{Existence of the critical front -- proof of Theorem \ref{t: existence}} \label{s: existence}
Our approach is to capture the weak exponential decay at $+\infty$ implied by the pinched double root by solving \eqref{e: EFKPP fronts} with an ansatz 
\begin{align}
q (x; \delta) = \chi_- (x) + w(x) + \chi_+ (x) (\mu + x) e^{-\eta_* (\delta) x}, \label{e: existence ansatz}
\end{align}
where $\chi_+$ is a smooth positive cutoff function satisfying 
\begin{align}
\label{e: chi plus}
\chi_+ (x) = \begin{cases}
1, &x \geq 3, \\
0, &x \leq 2,
\end{cases}
\end{align}
and $\chi_- (x) = \chi_+ (-x)$. For brevity, we denote by $\psi(\mu, \delta)$ the function
\begin{align*}
\psi(x; \mu, \delta) = (\mu + x) e^{-\eta_* (\delta) x}. 
\end{align*}
We will require $w$ to be exponentially localized, with a decay rate faster than $e^{-\eta_*(\delta) x}$ --- this localized piece is the \textit{core} of the solution, while $\chi_-$ and $\chi_+ \psi$ capture the \textit{far-field} behavior. Similar far-field/core decompositions have been used to construct heteroclinic solutions to pattern-forming systems in \cite{AveryGoh, GohScheel}. Inserting the ansatz \eqref{e: existence ansatz} into the traveling wave equation \eqref{e: EFKPP fronts}, we get an equation
\begin{align}
\mathcal{A}_+(\delta) (\chi_- + w + \chi_+ \psi(\mu, \delta)) + N(\chi_- + w + \chi_+ \psi(\mu, \delta)) = 0, \label{e: existence ff core unweighted}
\end{align}
where $\mathcal{A}_+ (\delta) = -\delta^2 \partial_x^4 + \partial_x^2 + c_*(\delta) \partial_x + f'(0)$, and $N(q) = f(q) - f'(0) q$. Since we want to require $w$ to decay faster than the front itself, we first let $v = \omega_* w$, so that \eqref{e: existence ff core unweighted} becomes
\begin{align*}
0 = F(v; \mu, \delta) := \mathcal{S}(\delta) v + \omega_* \mathcal{A}_+(\delta) (\chi_- + \chi_+ \psi) + \omega_* N(\chi_- + \omega_*^{-1} v + \chi_+ \psi),
\end{align*}
where $\mathcal{S}(\delta) = \omega_* \mathcal{A}_+(\delta) \omega_*^{-1}$ is the conjugate operator
\begin{align}
\mathcal{S}(\delta) = -\delta^2 \partial_x^4 + \delta^2 a_3 (x; \delta) \partial_x^3 + (1 + \delta^2 a_2 (x; \delta)) \partial_x^2 + a_1 (x; \delta) \partial_x + \tilde{a}_0 (x; \delta), \label{e: S def}
\end{align}
where the coefficients $a_i$ are given in \eqref{e: coefficients a} for $i = 1,2$ or $3$ while
\begin{align*}
\tilde{a}_0 = f'(0) + \omega_* \left(c_* \partial_x + \partial_x^2 - \delta^2 \partial_x^4 \right) \omega_*^{-1}, 
\end{align*}
since we are linearizing about the unstable state $u \equiv 0$ rather than the front itself, which we are in the process of constructing. 

Since $\omega_*(x; \delta) = 1$ on the support of $\chi_-$ and $\omega_*(x; \delta) = e^{\eta_*(\delta) x}$ on the support of $\chi_+$, we simplify $F$ to 
\begin{align*}
F(v; \mu, \delta) = \mathcal{S}(\delta) v + \mathcal{A}_+(\delta) \chi_- + \mathcal{S}(\delta) [(\mu + \cdot) \chi_+] + \omega_* N (\chi_- + \omega_*^{-1} v + \chi_+ \psi). 
\end{align*}
Then, we extract from $N$ terms that are linear in $v$, together with residual terms that are $v$-independent. We write
\begin{align*}
\omega_* N(\chi_- + \omega_*^{-1}v + \chi_+ \psi) = \mathcal{N}(v; \mu, \delta) + Q(\mu, \delta) v + R(\mu, \delta) 
\end{align*}
where
\begin{align}
\mathcal{N}(v; \mu, \delta) &=  \omega_* \left[f(\chi_+ + \omega_*^{-1} v + \chi_+ \psi) - f(\chi_- + \chi_+ \psi) - f'(\chi_- + \chi_+ \psi) \omega_*^{-1} v \right], \label{e: existence ff core nonlinearity}
\end{align}
and 
\begin{align*}
Q(\mu, \delta)v &= (f'(\chi_- + \chi_+ \psi)- f'(0)) v, 
\hspace{4em}
R(\mu, \delta) = \omega_* [f(\chi_- + \chi_+ \psi) - f'(0) (\chi_- + \chi_+ \psi)].
\end{align*}
Altogether, $F$ decomposes as the sum of a linear term, a residual term, and a nonlinear term:
\begin{align}
F(v; \mu, \delta) = [\mathcal{S}(\delta) + Q (\mu, \delta)] v + \mathcal{R}(\mu, \delta) + \mathcal{N} (v; \mu, \delta), \label{e: existence ff core F}
\end{align}
where $\mathcal{N}(v; \mu, \delta)$ is given by \eqref{e: existence ff core nonlinearity}, and 
\begin{align*}
\mathcal{R}(\mu, \delta) = R(\mu, \delta) + \mathcal{A}_+(\delta) \chi_- + \mathcal{S}(\delta) [(\mu + \cdot) \chi_+].  
\end{align*}
At $\delta = 0$, the equation $F(v; \mu, 0) = 0$ is the traveling wave equation for the Fisher-KPP equation, and so we have a solution $F(v_0; \mu_0, 0) = 0$ where 
\begin{align*}
v_0 = \omega_*(\cdot; 0) q_0 - \chi_- - \chi_+ \omega_*(\cdot; 0) \psi(\mu_0, 0),
\end{align*}
and $q_0$ is the translate of the critical Fisher-KPP front for which 
\begin{align*}
q_0 (x) = (\mu_0 + x)e^{-\eta_*(0)x} + \mathrm{O}(x^2 e^{-2\eta_*(0)x}), \quad x \to\infty,
\end{align*}
so that $v_0$ is exponentially localized (see e.g. \cite{Gallay} for asymptotics of the critical Fisher-KPP front).

To regularize the singular perturbation and enforce exponential localization of $v$, we consider
\begin{align*}
G(v; \mu, \delta) = (1- \delta^2 \partial_x^2)^{-1} F(v; \mu, \delta),
\end{align*}
as a nonlinear function $G: H^2_{0, \eta} (\R) \times \R \times \R\to L^2_{0, \eta} (\R)$, for $\eta > 0$ sufficiently small. 

\begin{lemma}\label{l: existence G regularity}
Fix $\eta > 0$ sufficiently small. There exists $\delta_0 > 0$ such that $(v, \mu, \delta) \mapsto G(v, \mu, \delta): H^2_{0, \eta} (\R) \times \R \times (-\delta_0, \delta_0) \to L^2_{0, \eta} (\R)$ is well-defined, smooth in $v$, and continuous in $\mu$ and $\delta$. Moreover, $\partial_v G$ and $\partial_\mu G$ are continuous in $\delta$. 
\end{lemma}
\begin{proof}
We use \eqref{e: existence ff core F} to write $G$ as
\begin{align}
G(v; \mu, \delta) = (1-\delta^2\partial_x^2)^{-1} \mathcal{S}(\delta) v + (1-\delta^2 \partial_x^2)^{-1} [Q(\mu, \delta) v + \mathcal{R}(\mu, \delta) + \mathcal{N}(v; \mu, \delta)]. \label{e: existence ff core well defined}
\end{align}
Using the fact that $f$ is smooth and that $H^2_{0, \eta} (\R)$ is a Banach algebra, one readily finds by Taylor expanding $f$ where it appears in $\mathcal{N}$ and $\mathcal{R}$ that if $v \in H^2_{0, \eta} (\R)$, then
\begin{align*}
||Q(\mu, \delta) v + R(\mu, \delta) + \mathcal{N}(v; \mu, \delta)||_{L^2_{0, \eta}} < \infty. 
\end{align*}
The remaining terms $A_+(\delta) \chi_-$ and $\mathcal{S}(\delta) [(\mu + \cdot) \chi_+]$ in $\mathcal{R}(\mu, \delta)$ are strongly localized by the choice of the far-field ansatz: $\chi_-(x)$ is identically zero for $x$ large, and for $x$ large every term in $\mathcal{S}(\delta)$ has at least two derivatives in it, so $\mathcal{S}(\delta) (\mu + \cdot) \equiv 0$ on the support of $\chi_+$, and the only terms that remain are compactly supported commutator terms. Hence we also obtain $\| \mathcal{R}(\mu, \delta) \|_{L^2_{0, \eta}} < \infty$. 

Together with $\eqref{e: precond L2 L2 bound}$ of Lemma \eqref{l: preconditioner estimates}, this implies that the second term of \eqref{e: existence ff core well defined} is in $L^2_{0, \eta} (\R)$, and so to check that $G$ is well-defined, it only remains to estimate the first term in \eqref{e: existence ff core well defined}. For this term, we use the specific form of $\mathcal{S}(\delta)$, given in \eqref{e: S def}, to write 
\begin{align}
(1-\delta^2 \partial_x^2)^{-1} \mathcal{S}(\delta) = \partial_x^2 + \delta^2 (1-\delta^2 \partial_x^2)^{-1} [a_3 \partial_x^3 + a_2 \partial_x^2] + (1-\delta^2 \partial_x^2)^{-1} (a_1 \partial_x + \tilde{a}_0). \label{e: existence ff core well defined S precond} 
\end{align}
Since $a_3$ and $a_2$ are smooth, constant outside of fixed compact set, and bounded uniformly in $\delta$, we have
\begin{align*}
||a_3 \partial_x^3 + a_2 \partial_x^2||_{H^2_{0, \eta} \to H^{-1}_{0, \eta}} \leq C. 
\end{align*}
Combining this with estimate \eqref{e: precond Hm Hmplus1 bound} of Lemma \ref{l: preconditioner estimates}, we obtain 
\begin{align}
||\delta^2 (1-\delta^2 \partial_x^2)^{-1} (a_3 \partial_x^3 + a_2 \partial_x^2)||_{H^2_{0, \eta} \to L^2_{0, \eta}} \leq C |\delta|. \label{e: O delta squared precondition}
\end{align}
The other terms in \eqref{e: existence ff core well defined S precond} are readily seen to be uniformly bounded in $\delta$ as operators from $H^2_{0, \eta} (\R)$ to $L^2_{0, \eta} (\R)$ for $\delta$ sufficiently small, from which we conclude that $G$ is well-defined. 

Since $f$ is smooth, smoothness in $v$ follows readily from the fact that $H^2_{0, \eta} (\R)$ is a Banach algebra whose norm controls the $L^\infty$ norm. Smoothness in $\mu$ is also readily attainable from smoothness of $f$ and the exponential localization of our ansatz. The preconditioner plays little role in these arguments --- when treating the residual terms or the nonlinearity, we do not need to use the preconditioner at all to obtain smoothness in $v$ and $\mu$. 

The residual terms as well as the nonlinearity are also readily seen to be continuous in $\delta$. The main subtlety is to handle the term $(1-\delta^2 \partial_x^2)^{-1} \mathcal{S}(\delta)$, which we write as 
\begin{align*}
(1-\delta^2)^{-1} \mathcal{S}(\delta) = (\partial_x^2 + a_1 \partial_x + \tilde{a}_0) + \delta^2(1- \delta^2 \partial_x^2)^{-1} (a_3 \partial_x^3 + a_2 \partial_x^2) + T(\delta) (a_1 \partial_x + \tilde{a}_0),
\end{align*}
where $T(\delta) = (1-\delta^2 \partial_x^2)^{-1} -1$. The operator $\partial_x^2 + a_1 (x, \delta) \partial_x + \tilde{a}_0 (x, \delta)$ is continuous in $\delta$ from $H^2_{0, \eta}$ to $L^2_{0, \eta}$, since the coefficients are smooth and uniformly bounded in $\delta$. The second term is continuous in $\delta$ by \eqref{e: O delta squared precondition}, and the last term is continuous in $\delta$ by Lemma \ref{l: T delta estimates}. Continuity in $\delta$ of $\partial_v G$ and $\partial_\mu G$ proceeds analogously. 

\end{proof}

With the appropriate regularity of $G$ in hand, we now aim to solve near $(v_0, \mu_0, 0)$ using the implicit function theorem. The linearization about this solution in $v$ is given by 
\begin{align*}
\partial_v G(v_0; \mu_0, 0) = \Cal{S}(0) + Q(\mu_0, 0) + \partial_v \mathcal{N} (v_0; \mu_0, 0) = \mathcal{S}(0) + f'(q_0) - f'(0) = \mathcal{L} (0). 
\end{align*}
From Lemma \ref{l: L0 fredholm properties}, $\partial_v G(v_0; \mu_0, 0)$ is Fredholm with index $-1$, so that the joint linearization $\partial_{(v, \mu)} G(v_0; \mu_0, 0)$ is Fredholm index 0 by the Fredholm bordering lemma \cite[Lemma 4.4]{ArndBjornRelativeMorse}. We show that in fact the joint linearization has full range, and hence is invertible.  
\begin{lemma}\label{l: existence linearization invertibility}
The joint linearization $\partial_{(v, \mu)} G(v_0; \mu_0, 0): H^2_{0, \eta} (\R) \times \R \to L^2_{0, \eta} (\R)$ is invertible. 
\end{lemma}
\begin{proof}
To show that $\partial_{(v, \mu)} G(v_0; \mu_0, 0)$ is invertible, we show that $\partial_\mu G(v_0; \mu_0, 0)$ is linearly independent from the range of $\mathcal{L} (0)$. From Lemma \ref{l: L0 range}, it is enough to obtain $\pscal{\partial_\mu G(v_0; \mu_0, 0), \varphi} \neq 0$. After a short computation, one finds
\begin{align*}
\partial_\mu G(v_0; \mu_0, 0) = \Cal{S}(0) \chi_+ + (f'(q_0) - f'(0)) \chi_+ = \mathcal{L}(0) \chi_+. 
\end{align*}
We compute $\langle \mathcal{L}(0) \chi_+, \varphi \rangle$ via integration by parts, with the goal being to move $\mathcal{L}(0)$ onto the other side of the inner product as its adjoint and exploit the fact that $\mathcal{L}(0)^* \varphi = 0$. However, we must be careful since $\varphi$ and $\chi_+$ are not localized at $\infty$, and in fact there is one boundary term from integration by parts which does not vanish. We see this by writing
\begin{align*}
\int_\R \chi_+'' \varphi \, dx = - \int_\R \chi_+' \varphi' \, dx = \int_\R \chi_+ \varphi'' - \left[ \chi_+ \varphi' \right]^{\infty}_{-\infty} = \langle \chi_+, \varphi'' \rangle - \varphi'(\infty) = \langle \chi_+, \varphi'' \rangle +\eta_*(0),
\end{align*}
where we have observed from Lemma \ref{l: L0 fredholm properties} that $\varphi'(\infty) = -\eta_*(0)$. Recalling that $\Cal{L}(0) = \partial_x^2 + f'(q_*)$ for $x\geq 1$, we obtain 
\begin{equation*}
\langle \mathcal{L}(0) \chi_+, \varphi \rangle = \langle \chi_+, \mathcal{L}(0)^* \varphi \rangle + \eta_*(0) = \eta_*(0) = \frac{c_*(0)}{2} > 0, 
\end{equation*}
which concludes the proof.
\end{proof}

\begin{proof}[Proof of Theorem \ref{t: existence}]
Since $G(v_0; \mu_0, 0) = 0$, $G$ is smooth in $v$ and $\mu$ and continuous in $\delta$ near $(v_0; \mu_0, 0)$, $\partial_{(v, \mu)} G(v_0; \mu_0, 0)$ is invertible, and $\partial_{(v, \mu)} G (v; \mu, \delta)$ is continuous in $\delta$, the implicit function theorem implies that for $\delta$ small, there exist $v(\delta) \in H^2_{0, \eta} (\R)$ and $\mu(\delta) \in \R$ depending continuously on $\delta$ near $\delta = 0$ such that $G(v(\delta); \mu(\delta), \delta) = 0$. By construction of $G$, this implies that 
\begin{align*}
q_*(x; \delta) := \chi_- (x) + \omega_*(x; \delta)^{-1} v(x; \delta) + \chi_+ (x) (\mu(\delta) + x) e^{-\eta_*(\delta) x} 
\end{align*}
solves \eqref{e: EFKPP fronts}. The claim that $q_*(\cdot, \delta) \to q_*(\cdot; 0) = q_0$ uniformly in space follows from the form of this ansatz, together with the fact that $H^2_{0, \eta} (\R)$  is continuously embedded in $L^\infty (\R)$.
\end{proof}

\section{Small eigenvalues}
\label{s: small eigenvalues}
Having established existence of the critical front, we are now ready to study the point spectrum of the linearization about the front. Here we show that there is no eigenvalue in a neighborhood of the origin, and in particular no resonance embedded in the essential spectrum at the origin. For this, we follow \cite{PoganScheel}: apply a Lyapunov-Schmidt reduction to construct a scalar function which vanishes at the eigenvalues, in a similar manner to the Evans function.

Throughout this section, we set $\Omega(\delta) := \lbrace 0 \rbrace \cup (\C \backslash \sigma_\Rm{ess}(\Cal{L}_\delta))$, and restrict to $\lambda \in \Omega(\delta)$. Then $\lambda$ is off the negative real axis, so that the principal value of $\gamma:= \sqrt{\lambda}$ is defined by $\real{\gamma} \geq 0$. 

\begin{prop}\label{p: gap lemma}
There exists $\delta_0, \gamma_0 > 0$ and a function $E: (-\delta_0, \delta_0) \times B(0, \gamma_0) \lra \C$, continuous in $\delta$ and analytic in $\gamma$ such that for all $\gamma \in \Omega(\delta)$, the eigenvalue problem 
\begin{equation}
\label{e: eigenvalue problem}
(\Cal{L}(\delta) - \gamma^2)u = 0
\end{equation}
admits a bounded solution $u$ if and only if $E(\delta, \gamma) = 0$. Furthermore, $E(0, 0) \neq 0$.
In particular, there exists $\gamma_1, \delta_1 > 0$ such that for all $\delta \in (-\delta_1, \delta_1)$, $\Cal{L}(\delta)$ has no eigenvalues on $B(0, {\gamma_1}^2)\cap \Omega(\delta_1)$.
\end{prop}
For any fixed $\delta\neq 0$, notice that \eqref{e: eigenvalue problem} is a linear, non degenerate ODE with smooth coefficients, so that any solution $u$ is smooth. Furthermore, such a solution admits exponential expansions at $\pm \infty$ (see the proof of Lemma \ref{l: ansatz} hereafter), so that when $\gamma^2$ is to the right of the essential spectrum, $u$ is bounded if and only if it lies in $H^4(\R)$, which is to say it is an eigenfunction. We will therefore consider bounded solutions from this point forward: for $\gamma^2$ to the right of the essential spectrum, they correspond with eigenfunctions, while at $\gamma = 0$ they capture resonances of $\mathcal{L}(\delta)$. 

We first show that a bounded solution of \eqref{e: eigenvalue problem} decomposes into two parts: a uniformly localized part, and a slowly decaying part, whose rate is $\gamma$-close to $0$.
\begin{lemma}
\label{l: ansatz}
Near $(\delta, \gamma) = (0,0)$, the roots of the polynomial $\nu\mapsto d^+_{c_*}(\gamma^2, - \eta_* + \nu)$ satisfy:
\begin{equation*}
\nu_1 = -\frac{1}{\absolu{\delta}} + \Rm{O}(1),
\hspace{2em}
\nu_2 =  -\gamma + \Rm{O}(\delta \gamma + \gamma^2),
\hspace{2em}
\nu_3 =  \gamma + \Rm{O}(\delta \gamma + \gamma^2),
\hspace{2em}
\nu_4 = \frac{1}{\absolu{\delta}} + \Rm{O}(1),
\end{equation*}
where each $\Rm{O}$ is taken as $\delta$ and $\gamma$ goes to $0$.

In particular, there exists $\delta_0>0$, $\gamma_0 > 0$ and $\eta > 0$ such that for all $\delta\in(-\delta_0,\delta_0)$, and $\gamma \in B(0, \gamma_0)$ with $\gamma^2 \in \Omega(\delta)$, a bounded and smooth solution $u$ of \eqref{e: eigenvalue problem} decompose as 
\begin{equation}
\label{e: LS ansatz}
u(x) = w(x) + \beta \chi_+(x) e^{\nu_2 x},
\end{equation}
where $w\in H^2_{0, \eta}(\R)$ and $\beta\in\C$. In this decomposition, $\chi_+$ is the cutoff function \eqref{e: chi plus}.
\end{lemma}
\begin{proof}
The claimed expansions of the four roots is purely technical and is postponed to the end of the proof.
Rewrite \eqref{e: eigenvalue problem} as a first order ODE in $\R^4$:
\begin{equation*}
\partial_x U = M(x; \delta, \gamma).
\end{equation*}
where $U = \transpose{(u, u', u'', u^{(3)})}$. The matrix $M$ converges towards $M_\pm(\delta, \gamma)$ when $x\to \pm\infty$, with an exponential rate which is independent of $\delta$ and $\gamma$. The eigenvalues of this asymptotic matrices $M_\pm$ are the roots of the dispersion relations $d_{c_*}^\pm(\gamma^2, -\eta_* + \cdot)$. It is standard that with such a convergence rate, these eigenvalues determine the behavior of $U$ at $\pm\infty$; see for example \cite[proof of Lemma 2.2]{FayeHolzer}.

More precisely, the behavior at $+\infty$ is the following. For $\gamma \neq 0$, the four roots are distinct, so that the exponential behavior is ensured: $U(x) \sim \sum_{i=1}^4 c_i(U) e^{\nu_i x}$ when $x\to +\infty$, with $c_i(U, \delta, \gamma)$ are vectors that does not depend on $x$. As $\gamma^2\notin \sigma_\Rm{ess}(\Cal{L}(\delta))$, the two small roots satisfy $\real{\nu_2(\delta, \gamma)} < 0 < \real{\nu_3(\delta, \gamma)}$, so that a bounded $U$ has exactly the claimed form.
At $\gamma = 0$, the two small roots merge to form a Jordan block. The proof in the above reference adapts, and we have the following expansion: $U(x) \sim c_1(U) e^{\nu_1 x} + c_2(U) + c_3(U) x + c_4(U) e^{\nu_4 x}$ when $x\to +\infty$. Once again the claimed decomposition is satisfied. 

At $-\infty$, the four roots of $d_{c_*}^-(\gamma^2, -\eta_* + \cdot)$ are distinct, and bounded away from $0$ with spectral gap uniform in $(\delta, \gamma)$. Then the expansion $U(x) \sim \sum_{i=1}^4 c_i^-(U) e^{\nu_i^- x}$ holds at $x\to -\infty$, so that any bounded $U$ lies in $H^2(\R_-)$. Hence the claimed decomposition holds. For an alternative argument not relying on the dynamical systems view of exponential expansions, see Remark \ref{rmk: evals are zeros of E}. 

We now establish the expansions of the roots by applying the implicit function theorem to $d_{c_*}^+$. 
From the choice of $\eta_*$ (see also \eqref{e: operator L+}) we have
\begin{equation*}
g_0(\delta, \gamma, \nu) := d_{c_*}^+(\gamma^2, -\eta_* + \nu) = - \delta^2 \nu^4 + 4\eta_* \delta^2 \nu^3 + (1-6\delta^2 \eta_*^2) \nu^2 - \gamma^2.
\end{equation*}
To avoid any $\delta$ singularity, we get rid of the $\delta^2$ in the dominant term by changing variables $\mu := \nu \absolu{\delta}$:
\begin{equation*}
g_1(\delta, \gamma, \mu) := \delta^2 g_0\big(\gamma, \delta, \frac{\mu}{\absolu{\delta}}\big) = - \mu^4 + 4 \eta_* \absolu{\delta} \mu^3 + (1 - 6 \delta^2 \eta_*^2)\mu^2 - \gamma^2 \delta^2.
\end{equation*}
At $(\delta, \gamma) = (0,0)$, this reduces to $g_1(0,0,\mu) = -\mu^2(\mu-1)(\mu+1)$. Applying the implicit function theorem to the simple root $-1$, we construct a root $\mu_1(\delta, \gamma)$ for $g_1(\delta, \gamma, \cdot)$ whose derivatives can be computed iteratively by differentiating the relation $g_1(\delta, \gamma, \mu_1(\delta, \gamma)) = 0$.
One can show by induction that any pure derivative in $\gamma$ is null: $\partial_\gamma^k \mu_1(0,0) = 0$ for $k\in \N^*$. This ensures that the Taylor expansion has the form
\begin{equation*}
\mu_1(\delta, \gamma) = -1 - \delta \frac{\partial_\delta g_1(0,0, -1)}{\partial_\mu g_1(0,0, -1)} - \gamma \frac{\partial_\gamma g_1(0,0, -1)}{\partial_\mu g_1(0,0, -1)} + \Rm{O}(\delta^2 + \delta\gamma) = -1 + \Rm{O}(\delta).
\end{equation*}
Coming back to the original variable, we define $\nu_1(\delta, \gamma) = \mu_1(\delta, \gamma) / \absolu{\delta}$, which satisfies the claimed expansion. The same steps can be applied to define $\mu_4(\delta, \gamma) = 1 + \Rm{O}(\delta)$, which in turn leads to $\nu_4(\delta, \gamma)$ as claimed.

To unfold the double root at $\mu = 0$, we change variables once again to $\nu = \gamma\sigma$:
\begin{equation*}
g_2(\delta, \gamma, \sigma) = \frac{g_0(\delta, \gamma, \gamma \sigma)}{\gamma^2} = - \delta^2 \gamma^2 \sigma^4 + 4\eta_* \delta^2 \gamma \sigma^3 + (1-6\delta^2 \eta_*^2) \sigma^2 - 1.
\end{equation*}
At $(\delta, \gamma) = (0,0)$, this reduces to $g_2(0, 0, \sigma) = (\sigma - 1)(\sigma + 1)$. Applying the implicit function theorem once again gives rise to 
\begin{equation*}
\sigma_2(\delta, \gamma) = -1 + \Rm{O}(\delta+\gamma),
\hspace{4em}
\sigma_3(\delta, \gamma) = 1 + \Rm{O}(\delta+\gamma),
\end{equation*}
which in turns leads to the claimed estimates on $\nu_2(\delta, \gamma) = \gamma \sigma_2(\delta, \gamma)$ and $\nu_3(\delta, \gamma) = \gamma \sigma_3(\delta, \gamma)$.
\end{proof}

As in the existence of the critical front, our problem is singular at $\delta = 0$. Hence, we apply the same preconditioner: when $\delta$ is small, \eqref{e: eigenvalue problem} is equivalent to 
\begin{equation*}
\label{e: eigenvalue problem reg}
(1 - \delta^2 \partial_x^2)^{-1} \, (\Cal{L}(\delta) - \gamma^2)u = 0.
\end{equation*}
We now use the decomposition of Lemma \ref{l: ansatz} to separate out the localized part of our problem from the far-field behavior, which will allow us to make use of the Fredholm properties on weighted spaces of Section \ref{s: Fred Prop L0}. In the following, for $\delta\in (-\delta_0, \delta_0)$ and $\gamma \in B(0,\gamma_0)$ we let 
\begin{equation*}
A(\delta, \gamma, \eta) = \lbrace w + \beta \chi_+ e^{\nu_2(\delta, \gamma)\cdot} : w \in H^2_{0,\eta}(\R), \beta \in\R \rbrace,
\end{equation*}
denote the set where the ansatz obtained above holds.

\begin{lemma}
\label{l: LS localization}
There exist positive constants $\delta_0$, $\gamma_0$ and $\eta$ such that if  $\delta\in(-\delta_0, \delta_0)$, $\gamma \in B(0, \gamma_0)$ and $u\in A(\delta, \gamma, \eta)$, then $(1 - \delta^2 \partial_x^2)^{-1} \, (\Cal{L}(\delta) - \gamma^2) u \in L^2_{0,\eta}(\R)$.
\end{lemma}
\begin{proof}
First, $(1-\delta^2 \partial_x^2)^{-1}(\Cal{L}(\delta) - \gamma^2) w$ belongs to $L^2_{0,\eta}(\R)$ by the choice of the preconditioner, using the same regularization effect we observed in \eqref{e: existence ff core well defined S precond}. Then, as $\chi_+$ is smooth, vanishes on $({-}\infty, 2)$ and is constant on $(3, +\infty)$, it only remains to show that $(\Cal{L}(\delta) - \gamma^2) e^{\nu_2 \cdot} \in L^2_{0,\eta}(\R_+)$.
For $x\geq 1$, almost all coefficients of $\Cal{L}$ are constants, see \eqref{e: coefficients a}, hence we compute
\begin{equation*}
(\Cal{L}(\delta) -\gamma^2) e^{\nu_2 x} = (\Cal{L}(\delta) - \Cal{L}_+(\delta)) e^{\nu_2 x} + (\Cal{L}_+(\delta) -\gamma^2) e^{\nu_2 x} = \left(f'(q_*(x; \delta)) - f'(0)\right) e^{\nu_2 x} + P(\nu_2,\delta, \gamma) e^{\nu_2 x},
\end{equation*}
where the polynomial $P(X, \delta, \gamma)$ is the symbol defined by: $\Cal{L}_+(\delta) - \gamma^2 = P(\partial_x, \delta, \gamma)$, and $\Cal{L}_+$ is the asymptotic operator \eqref{e: operator L+}. From the definition of $\nu_2(\delta, \gamma)$, $P(X, \delta, \gamma)$ vanishes at $X=\nu_2(\delta, \gamma)$, hence for $x\geq 1$:
\begin{equation*}
(\Cal{L}(\delta) - \gamma^2)e^{\nu_2 x} = (f'(q_*(x)) - f'(0))e^{\nu_2 x} = f''(0) q_*(x; \delta) e^{\nu_2 x} + \Rm{O}\left(e^{\nu_2 x}{q_*(x; \delta)}^2\right).
\end{equation*}
The right hand side belongs to $L^2_{0,\eta}(\R)$ as long as $\eta$ satisfies
\begin{equation}
\label{e: LS condition negative growth}
-\eta_* + \real{\nu_2(\delta, \gamma)} < -\eta.
\end{equation}
We can take a smaller $\gamma_0$ than in Lemma \ref{l: ansatz}, so that $\sup_{\delta, \gamma} \lbrace -\eta_*(\delta) + \real{\nu_2(\delta, \gamma)}\rbrace < 0$, which then allows to fix $\eta> 0$ so that \eqref{e: LS condition negative growth} is satisfied for all $\delta\in(-\delta_0, \delta_0)$ and $\gamma\in B(0, \gamma_0)$. This concludes the proof.
\end{proof}
We can now use Lemma \ref{l: L0 range} to decompose our problem into a part which belongs to $\range{\Cal{L}(0)}$ and a complementary part. Recall that $P : L^2_{0,\eta}(\R) \lra \range(\Cal{L}(0))$ and that $\varphi$ allows to describe $\range(\Cal{L}(0))$.
Fix $\delta\in(-\delta_0, \delta_0)$, and $\gamma \in B(0, \gamma_0)\cap \Omega(\delta)$. If $u$ is a bounded solution of \eqref{e: eigenvalue problem} then $(w, \beta) \in H^2_{0, \eta}(\R)\times \C$ defined in Lemma \ref{l: ansatz} solves:
\begin{equation}
\left\lbrace
\begin{array}{l}
P \,(1 - \delta^2\partial_x^2)^{-1}\,(\Cal{L}(\delta) - \gamma^2)(w + \beta h) = 0, \\[0.5em]
\pscal{(1 - \delta^2\partial_x^2)^{-1}\,(\Cal{L}(\delta) - \gamma^2)(w + \beta h), \varphi} = 0,
\end{array}
\right. \label{e: LS preconditioned}
\end{equation}
where $h(x) = \chi_+(x) e^{\nu_2(\gamma)x}$. Reciprocally, if $(w, \beta)\in H^2_{0, \eta}(\R)\times \C$ satisfies \eqref{e: LS preconditioned}, then $u = w + \beta h$ is bounded and satisfies \eqref{e: eigenvalue problem}. We write the first equation as 
\begin{align}
0 = \mathcal{F}(w, \beta; \gamma, \delta) := P \,(1 - \delta^2\partial_x^2)^{-1}\,(\Cal{L}(\delta) - \gamma^2)(w + \beta h). \label{e: LS precond w equation}
\end{align}
and solve it with the implicit function theorem. We will then use the second equation to define $E(\delta, \gamma)$.

\begin{lemma}\label{l: LS w regularity}
For $\eta >0$ sufficiently small, the map $\mathcal{F}: H^2_{0, \eta} (\R) \times \C \times B(0, \gamma_0) \times (-\delta_0, \delta_0) \to L^2_{0, \eta} (\R)$ is smooth in $w$ and $\beta$, analytic in $\gamma$, and continuous in $\delta$. Moreover, $\partial_w \mathcal{F}(w, \beta; \gamma, \delta)$ is continuous in $\beta$, $\gamma$, and $\delta$. 
\end{lemma}

\begin{proof}
Note that $\mathcal{F}$ is linear in $w$ and $\beta$, so smoothness is automatic provided the linear part in $w$ is well defined, which is guaranteed here by
Lemma \ref{l: LS localization}. 
For the continuity of $\partial_w \mathcal{F}(w, \beta; \gamma, \delta)$, we write
\begin{align*}
(1-\delta^2 \partial_x^2)^{-1} \mathcal{L}(\delta) = \partial_x^2 + a_1 \partial_x + a_0 + \delta^2 (1- \delta^2 \partial_x^2)^{-1} (a_3 \partial_x^3 + a_2 \partial_x^2) + T(\delta) (a_1 \partial_x + a_0).
\end{align*}
We see by Lemmas \ref{l: preconditioner estimates} and \ref{l: T delta estimates} that $(1-\delta^2 \partial_x^2)^{-1}\mathcal{L}(\delta)$ is a a well-defined family of bounded operators from $H^2_{0, \eta}$ to $L^2_{0, \eta}$, depending continuously on $\delta$. This is of course preserved when we compose with the projection $P$. We write the other term in the linearization in $w$ as 
\begin{align*}
\gamma^2 (1-\delta^2 \partial_x^2)^{-1} w = \gamma^2 w +  \gamma^2 T(\delta) w, 
\end{align*}
which is again continuous in $\gamma$ and $\delta$ as a bounded linear operator from $H^2_{0, \eta}$ to $L^2_{0, \eta}$ by Lemma \ref{l: T delta estimates}. Hence $\partial_w \mathcal{F}$ is continuous in its three last variables.
Analyticity of $\Cal{F}$ in $\gamma$ follows as in \cite[Proposition 5.11]{PoganScheel}. For the continuity of $\Cal{F}$ with respect to $\delta$, it only remains to look at the terms associated to $h$. We rewrite 
\begin{align*}
(\mathcal{L}(\delta) - \gamma^2) h = [\mathcal{L}_+ (\delta), \chi_+] e^{\nu_2 (\delta, \gamma) \cdot} + (\mathcal{L}(\delta) - \mathcal{L}_+ (\delta)) e^{\nu_2(\delta, \gamma) \cdot},
\end{align*}
using the fact that $(\mathcal{L}_+(\delta)-\gamma^2)e^{\nu_2(\delta, \gamma) \cdot} = 0$, and where $[\mathcal{L}_+ (\delta), \chi_+] = \mathcal{L}_+(\delta) (\chi_+ \cdot) - \chi_+ \mathcal{L}_+(\delta)$ is the commutator between these operators. In this form, we recognize that $[\mathcal{L}_+ (\delta), \chi_+]$ and $(\mathcal{L}(\delta) - \mathcal{L}_+(\delta))$ are both differential operators with exponentially localized coefficients, with rate uniform in $\delta$ for $\delta$ small. By Lemma \ref{l: ansatz}, $e^{\nu_2 (\delta, \gamma) x}$ is continuous in $\gamma$ and $\delta$ for each fixed $x$, and the uniform localization of $[\mathcal{L}_+(\delta), \chi_+]$ and $(\mathcal{L}(\delta) - \mathcal{L}_+(\delta))$ guarantees that these terms are continuous in $\delta$ in $L^2_{0, \eta}$ for $\eta$ small. In fact, since $h$ is a smooth function, we see that $\delta \mapsto (\mathcal{L}(\delta) - \gamma^2) h$ is in particular continuous from $(-\delta_1, \delta_1)$ to $H^1_{0, \eta}$. Taking into account the preconditioner, we write
\begin{align*}
(1-\delta^2 \partial_x^2)^{-1} (\mathcal{L}(\delta) - \gamma^2)h = (\mathcal{L}(\delta) - \gamma^2) h + T(\delta) (\mathcal{L}(\delta) - \gamma^2) h. 
\end{align*}
By Lemma \ref{l: T delta estimates}, this term is continuous in $\delta$, as desired. 
\end{proof}

\begin{corollary} \label{c: LS w linear in beta}
For $\gamma, \delta$ sufficiently small, and for $\beta \in \C$, there is a family of solutions $w$ to \eqref{e: LS precond w equation} which have the form
\begin{equation}
w(\beta; \gamma, \delta) = \beta \tilde{w} (\gamma, \delta). \label{e: LS w linear in beta}
\end{equation}
Moreover, any solution to \eqref{e: LS precond w equation} with $\gamma, \delta$ small has this form. 
\end{corollary}
\begin{proof}
We begin with the trivial solution $\mathcal{F}(0, 0; 0, 0) = 0$. The linearization in $w$ about this trivial solution is $\partial_w \mathcal{F}(0, 0; 0, 0) = P \mathcal{L}(0)$, which is invertible by Lemmas \ref{l: L0 fredholm properties} and \ref{l: L0 range}. Together with Lemma \ref{l: LS w regularity}, this implies that we can solve near this trivial solution with the implicit function theorem, obtaining a unique solution $w(\beta; \gamma, \delta)$ in a neighborhood $\mathcal{U}$ of $(0, 0; 0, 0)$. Since \eqref{e: LS precond w equation} is linear in $w$ and $\beta$, by uniqueness any solution in this neighborhood can be written as 
\begin{equation*}
w(\beta; \gamma, \delta) = \beta \tilde{w} (\gamma, \delta) 
\end{equation*}
for some function $\tilde{w}(\gamma, \delta) \in H^2_{0, \eta} (\R)$. 

If for some fixed $\gamma$, $\delta$ small we have another solution $(w_0, \beta_0)$ to \eqref{e: LS precond w equation} which does not a priori have this form, by dividing by a sufficiently large constant $K(||w_0||_{H^2_{0, \eta}}, \beta)$ we get another solution which belongs to the neighborhood $\mathcal{U}$ where we have solved with the implicit function theorem, and so we conclude that 
\begin{equation*}
\frac{w_0}{K(||w_0||_{H^2_{0, \eta}}, \beta)} = \frac{\beta}{K(||w_0||_{H^2_{0, \eta}}, \beta)} \tilde{w}(\gamma, \delta), 
\end{equation*}
and hence the solution $(w_0, \beta_0)$ in fact has the form \eqref{e: LS w linear in beta}, as claimed. 
\end{proof}

Having solved the first equation in \eqref{e: LS preconditioned} with the implicit function theorem, we now insert this solution $w(\beta; \gamma, \delta) = \beta \tilde{w}(\gamma, \delta)$ into the second equation, so that \eqref{e: LS preconditioned} has a solution if and only if 
\begin{align}
0 = E(\delta, \gamma) := \langle (1-\delta^2 \partial_x^2)^{-1} (\mathcal{L} (\delta) - \gamma^2) (\tilde{w}(\gamma, \delta) + h), \varphi \rangle. \label{e: E definition}
\end{align}
Note that we have been able to eliminate the $\beta$ dependence in this equation, since all terms in this equation are linear in $\beta$ by Corollary \ref{c: LS w linear in beta}. Since the projection $P$ played no role in the proof of Lemma \ref{l: LS w regularity}, the same argument shows that $E$ is continuous in both of its arguments. 

\begin{lemma}
The function $E: (-\delta_0, \delta_0) \times B(0, \gamma_0) \to \C$ is continuous in both arguments, and analytic in $\gamma$ for fixed $\delta$. 
\end{lemma}

\begin{proof}[Proof of Proposition \ref{p: gap lemma}]
It only remains to prove that $E(0, 0) \neq 0$. From \eqref{e: E definition}, we see that 
\begin{align*}
E(0, 0) = \langle \mathcal{L}(0) (\tilde{w} (0, 0) + \chi_+), \varphi \rangle. 
\end{align*}
Since $\tilde{w}(0, 0) \in H^2_{0, \eta} (\R)$, we have
\begin{align*}
\langle \mathcal{L}(0) \tilde{w}(0, 0), \varphi \rangle = \langle \tilde{w}(0, 0), \mathcal{L}(0)^* \varphi \rangle = 0. 
\end{align*}
Hence we obtain 
\begin{align*}
E(0, 0) = \langle \mathcal{L} (0) \chi_+, \varphi \rangle = \eta_*(0) \neq 0,
\end{align*}
by the computation in the proof of Lemma \ref{l: existence linearization invertibility}. 
\end{proof}

\begin{rmk}\label{rmk: evals are zeros of E} Rather than using the spatial dynamics approach to exponential expansions outlined in the proof of Lemma \ref{l: ansatz} to show directly that eigenfunctions have the form \eqref{e: LS ansatz}, one can instead show that for $\delta \neq 0$, $(\mcl(\delta) - \gamma^2): H^4 (\R) \to L^2(\R)$ is invertible if $E(\delta, \gamma) \neq 0$ and $\gamma^2$ is to the right of the essential spectrum of $\mcl(\delta)$, using an argument adapted from \cite{PoganScheel}. Indeed, if $\gamma^2$ is to the right of the essential spectrum, then $\mcl(\delta) - \gamma^2$ is Fredholm index 0, and in particular has closed range, so to invert this operator on $L^2(\R)$, it suffices to solve $(\mcl(\delta) - \gamma^2) u = g$ for $g$ in the dense subspace $L^2_{0, \eta} (\R)$. The fact that the range of $\mcl(\delta) - \gamma^2$ is closed then implies $\mcl(\delta) - \gamma^2$ is surjective, and hence invertible since it is Fredholm of index 0. The open mapping theorem then implies the inverse is bounded, so $\gamma^2$ will be in the resolvent set of $\mcl(\delta)$. To solve $(\mcl(\delta) - \gamma^2) u = g$ for $g \in L^2_{0, \eta} (\R)$, one looks for solutions in the form \eqref{e: LS ansatz}, and finds that $(w,\beta)$ solve the system \eqref{e: LS preconditioned} but with $(0,0)^T$ on the right hand side replaced by $(Pg, \langle g, \phi \rangle)^T$. We can always solve the first equation with the implicit function theorem, and we can solve the second equation precisely when $E(\delta, \gamma) \neq 0$, as claimed. At $\gamma = 0$ we lose Fredholm properties on $L^2(\R)$, but the fact that $E(\delta, 0) \neq 0$ implies there is no solution to $\mcl(\delta) u = 0$ of the form $u = w + \beta \chi_+$ for $w$ exponentially localized, and this is actually all that is needed in \cite{AveryScheel} to prove nonlinear stability. One could additionally use a modified far-field/core decomposition at $\gamma = 0$ to prove that all bounded solutions to $\mcl(\delta) u = 0$ have the form $u = w+ \beta \chi_+$.  

\end{rmk}

\section{Large and intermediate eigenvalues --- proof of Theorem \ref{t: spectral stability}} \label{s: large eigenvalues}

Here, we conclude the study of the point spectrum. We first exclude any large unstable point spectrum, using mostly that the operator is sectorial.

\begin{figure}
\hspace{\stretch{1}}
\includegraphics[scale=\figurescale]{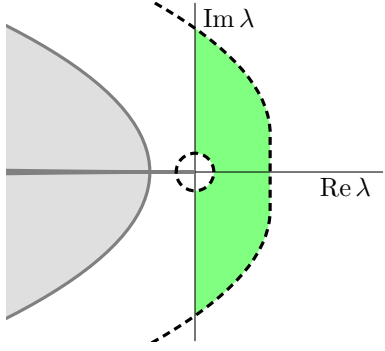}
\hspace{\stretch{1}}
\caption{Three regions for the study of the point spectrum. The function $E(\delta, \gamma)$ from Section \ref{s: small eigenvalues} rules out point spectrum in the dashed ball centered at the origin, together with a potential eigenvalue at the origin. Proposition \ref{p: large eigenvalues} excludes point spectrum to the right of the dashed curve. Finally, the green region to the right contains no point spectrum provided $\delta$ is small enough, see Proposition \ref{p: moderate eigenvalues}.}
\label{f: point spectrum}
\end{figure}
\begin{prop}
\label{p: large eigenvalues}
There exists a compact set $K\subset \C$ such that for all $\delta$ small, any eigenvalue $\lambda$ of $\Cal{L}(\delta)$ with $\real{\lambda} \geq 0$ lies in $K$. More precisely, an eigenvalue $\lambda$ satisfies:
\begin{equation*}
\real{\lambda} \leq \norme{b(\cdot; \delta)}_\infty, 
\hspace{4em}
\absolu{\imag{\lambda}} \leq c_* \sqrt{\norme{b(\cdot; \delta)}_\infty - \real{\lambda}},
\end{equation*}
where $b(\cdot\,;\delta) = f'(q_*(\cdot\, ; \delta))$ is uniformly bounded.
\end{prop}
\begin{proof}
We work with $\delta\in(-\delta_0, \delta_0)$, with $\delta_0$ small enough so that Theorem \ref{t: existence} applies.
Assume that $\lambda\in\C$ and $\psi \in H^4(\R)$ satisfy $\Cal{L}(\delta)\psi = \lambda \psi$. Coming back to the unweighted operator $\Cal{A}(\delta) = \omega_*^{-1} \Cal{L}(\delta) \omega_*$, defined by \eqref{e: unweighted linearized EFKPP}, we obtain 
\begin{equation}
\label{e: large eig 1}
\Cal{A}(\delta) \phi = \lambda \phi
\end{equation}
with $\phi = \omega_*^{-1} \psi \in H^4(\R)$. Up to a scalar multiplication, we can assume that $\norme{\phi}_{L^2(\R)} = 1$.
Now we take the $L^2(\R)$-inner product of \eqref{e: large eig 1} with $\phi$, and pass into Fourier space, to obtain that:
\begin{equation}
\label{e: large eig 2}
\int_\R (-\delta^2 \xi^4 - \xi^2 + ic_* \xi) \absolu{\hat{\phi}(\xi)}^2 \d \xi + \int_\R f'(q_*(x)) \absolu{\phi(x)}^2 \d x = \lambda,
\end{equation}
where $\mathcal{F} u = \hat{u}$ denotes the Fourier transform of a function $u$. We let $I_0 = \int_\R f'(q_*(x)) \absolu{\phi(x)}^2 \d x$ denote the 0th order term. Then, real and imaginary parts of equation \eqref{e: large eig 2} give
\begin{equation*}
\real{\lambda} - I_0 = -\int_{\R}(\delta^2 \xi^4 + \xi^2) \absolu{\hat{\phi}}^2 \d \xi,
\hspace{4em}
\imag{\lambda} = c_*\int_{\R} \xi \absolu{\hat{\phi}}^2 \d \xi.
\end{equation*}
Hence by the Cauchy-Schwartz inequality:
\begin{equation}
\label{e: large eig 4}
0 \leq \left(\imag{\lambda}\right)^2 \leq c_*^2 \norme{\hat{\phi}}_{L^2}^2 \int_\R \xi^2 \absolu{\hat{\phi}}^2 \d \xi \leq  c_*^2\int_{\R}(\delta^2 \xi^4 + \xi^2) \absolu{\hat{\phi}}^2 \d \xi = c_*^2 \left({I_0} - \real{\lambda}\right).
\end{equation}
Note that $b(\cdot\, ; \delta)$ is uniformly bounded with respect to $\delta \in (-\delta_0,\delta_0)$, since this holds for $q_*(\cdot\, ; \delta)$ from Theorem \ref{t: existence}, and since $f'$ is continuous. 

Observe that $\absolu{I_0} \leq \norme{b}_\infty \norme{\phi}_{L^2}^2$. Inserting this into \eqref{e: large eig 4} leads to the claimed bounds on $\real{\lambda}$ and $\imag{\lambda}$. 
These bounds together with the requirement $\real{\lambda} \geq 0$ define a compact set $K$.
\end{proof} 

We now conclude the proof of Theorem \ref{t: spectral stability} by excluding the possibility of any eigenvalues in the intermediate region; see Figure \ref{f: point spectrum}.
\begin{prop}\label{p: moderate eigenvalues}
For each $\delta_0> 0$ sufficiently small, there exists $r(\delta_0) > 0$ with $r(\delta_0) \to 0$ as $\delta_0 \to 0$ such that for all $\delta$ with $|\delta| < \delta_0$, the operator $\mathcal{L}(\delta)$ has no eigenvalues in $\{ \real{\lambda} \geq 0 \} \setminus B(0, r(\delta_0))$. 
\end{prop} 
\begin{proof}
Suppose to the contrary that there exists a sequence $\delta_n \to 0$ with corresponding eigenvalues $\lambda_n$ bounded away from the origin, with $\real{\lambda_n} \geq 0$, and with eigenfunctions $u_n$. We normalize the eigenfunctions so that $||u_n||_{H^2} = 1$ for all $n$. By Proposition \ref{p: large eigenvalues}, these eigenvalues all belong to the compact set $K$. By compactness, we extract a subsequence along which $\lambda_n \to \lambda_\infty$ for some $\lambda_\infty$ with $\lambda_\infty \in K$, and $\lambda_\infty \neq 0$, since the sequence was bounded away from the origin. We now show that in this limit, $\lambda_\infty$ is an eigenvalue for $\mathcal{L}(0)$ with $\real{\lambda} \geq 0$, contradicting the spectral stability of this operator.

These eigenfunctions solve $(\mathcal{L}(\delta_n) - \lambda_n) u_n = 0$. We precondition by applying $(1-\delta_n^2 \partial_x^2)^{-1}$ to both sides of this equation, obtaining
\begin{align}
\left[ \partial_x^2 + a_1(\cdot, \delta_n) \partial_x + a_0(\cdot, \delta_n) + E_1 (\delta_n) + E_2(\delta_n) - \lambda_n  + \lambda_n T(\delta_n)\right] u_n = 0, \label{e: moderate eval eqn}
\end{align}
where 
\begin{align*}
E_1(\delta_n) = \delta_n^2 (1-\delta_n^2 \partial_x^2)^{-1} (a_3(\cdot, \delta_n) \partial_x^3 + a_2 (\cdot, \delta_n) \partial_x^2), 
\hspace{4em}
E_2 (\delta_n) = T(\delta_n) (a_1 (\cdot, \delta_n) \partial_x + a_0(\cdot, \delta_n)). 
\end{align*}
We relate this to the KPP linearization $\mathcal{L}(0)$ by rewriting \eqref{e: moderate eval eqn} as 
\begin{align*}
(\mathcal{L}(0)-\lambda_\infty) u_n = - E_1 (\delta_n) u_n - E_2 (\delta_n) u_n + E_3 (\delta_n) u_n + (\lambda_n - \lambda_\infty) u_n + \lambda_n T(\delta_n) u_n =: f_n
\end{align*}
where
\begin{align*}
E_3 (\delta_n) = (a_1(\cdot, 0) - a_1(\cdot, \delta_n)) \partial_x + (a_0 (\cdot, 0) - a_0 (\cdot, \delta_n)).
\end{align*}
It follows from Lemma \ref{l: preconditioner estimates} and the fact that the coefficients $a_j(\cdot; \delta)$ are uniformly bounded in $\delta$ that
\begin{align*}
||E_1(\delta_n) u_n||_{L^2} \leq C \delta_n ||u_n||_{H^2} = C \delta_n. 
\end{align*} 
Similarly, by Lemma \ref{l: T delta estimates} we see that $E_2(\delta_n) u_n \to 0$ and $\lambda_n T(
\delta_n) u_n \to 0$ in $L^2$ as $n \to \infty$, since $\lambda_n$ and $u_n$ are uniformly bounded in $n$. Lastly, by the construction of the exponential weights, the fact that $q_* (\cdot; \delta)$ converges uniformly to $q_*(\cdot; 0)$ as $\delta \to 0$ by Theorem \ref{t: existence}, and the assumption that $||u_n||_{H^2}$ is uniformly bounded, we see that also $E_3(\delta_n) u_n \to 0$ in $L^2$ as $n \to \infty$. Hence $f_n$ converges to zero in $L^2$ as $n \to \infty$. 

Since $\lambda_\infty$ is not in the spectrum of $\mathcal{L}(0)$, we can can invert $(\mathcal{L}(0)-\lambda_\infty)$ to write
\begin{align*}
u_n = (\mathcal{L}(0) - \lambda_\infty)^{-1} f_n,
\end{align*}
from which we observe that $u_n \to 0$ in $H^2(\R)$ as $n \to \infty$ by boundedness of the resolvent operator. This is a contradiction since we have normalized $u_n$ so that $||u_n||_{H^2} = 1$. 
\end{proof}

\begin{proof}[Proof of Theorem \ref{t: spectral stability}]
By Proposition \ref{p: gap lemma}, there exist $\gamma_1, \delta_1 > 0$ so that for all $\delta \in (-\delta_1, \delta_1)$, $\mathcal{L}(\delta)$ has no eigenvalues in $B(0, \gamma_1^2)$, and also has no resonance at $\lambda=  0$. By Proposition \ref{p: moderate eigenvalues}, there exists a $\delta_0 > 0$ so for all $\delta \in (-\delta_0, \delta_0)$, $\mathcal{L}(\delta)$ has no eigenvalues in $\{ \real{\lambda} \geq 0 \} \setminus B(0, \frac{\gamma_1^2}{2})$. Hence for all $\delta \in (-\delta_0, \delta_0)$, $\mathcal{L}(\delta)$ has no eigenvalues in $\{ \real{\lambda} \geq 0 \}$, as desired. 
\end{proof}

\bibliographystyle{plain}

\bibliography{references}

\end{document}